\documentclass[final,3p]{elsarticle}
 \usepackage{graphics}
 \usepackage{graphicx}
 \usepackage{epsfig}
\usepackage{amssymb}
 \usepackage{amsthm}
 \usepackage{lineno}
 \usepackage{amsmath}
   \numberwithin{equation}{section}
\usepackage{mathrsfs}

\newtheorem{thm}{Theorem}[section]
\newtheorem{cor}[thm]{Corollary}
\newtheorem{lem}[thm]{Lemma}

\newtheorem{defn}[thm]{Definition}
\newtheorem{rem}[thm]{Remark}

\biboptions{sort&compress,square}
\begin{document}
\begin{frontmatter}
\author[rvt1,rvt2]{Jian Wang}
\author[rvt1]{Yong Wang\corref{cor2}}
\ead{wangy581@nenu.edu.cn}
\cortext[cor2]{Corresponding author.}
\address[rvt1]{School of Mathematics and Statistics, Northeast Normal University,
Changchun, 130024, P.R.China}
\address[rvt2]{Chengde Petroleum College,
Chengde, 067000, P.R.China}

\title{Noncommutative Residue  and  Dirac operators for Manifolds \\ with the Conformal Robertson-Walker metric}
\begin{abstract}
In this paper, we prove a Kastler-Kalau-Walze type theorem for
4-dimensional and 6-dimensional spin manifolds with boundary  associated with the conformal Robertson-Walker metric. And we give
two kinds of operator theoretic explanations of the gravitational action for boundary in the case of 4-dimensional manifolds with
flat boundary. In particular,  for $6$-dimensional spin manifolds with boundary with the conformal Robertson-Walker metric, we obtain the
 noncommutative residue of the composition of $\pi^+D^{-1}$ and $\pi^+D^{-3}$ is proportional to the Einstein-Hilbert action for
 manifolds with boundary.
\end{abstract}
\end{frontmatter}
\section{Introduction}
The noncommutative residue plays a prominent role in noncommutative geometry \cite{Gu}\cite{Wo}.
 In \cite{Co1}, Connes used the noncommutative residue to derive a conformal 4-dimensional Polyakov action analogy.  Several
years ago, Connes made a challenging observation that the noncommutative residue of the square of the inverse of the Dirac
operator was proportional to the Einstein-Hilbert action, which was called Kastler-Kalau-Walze Theorem now. In \cite{Ka}, Kastler gave a
brute-force proof of this theorem. In \cite{KW}, Kalau and Walze proved this theorem by the normal coordinates way simultaneously.
In \cite{Ac}, Ackermann gave a note on a new proof of this theorem by means of the heat kernel expansion.

On the other hand, Fedosov etc. defined a noncommutative residue on Boutet de Monvel's algebra and proved that it was a
unique continuous trace in \cite{FGLS}. Wang generalized  the Connes' results to the case of manifolds with boundary in \cite{Wa1} \cite{Wa2} ,
and proved a Kastler-Kalau-Walze type theorem for the Dirac operator and the signature operator for lower-dimensional manifolds
with boundary.  In \cite{WW}, we get the Kastler-Kalau-Walze type theorem associated to nonminimal operators by heat equation asymptotics
on compact manifolds without boundary. The purpose of papers \cite{Wa3} \cite{Wa4} is to prove  the Kastler-Kalau-Walze Theorem for manifolds
 associated with the metric near the boundary as follows: $ g^M=\frac{1}{h(x_n)}g^{\partial M}+dx_n^2$.
  In \cite{AF}, Antoci considered the metric $ g^M=e^{-2(a+1)x_{n}}dx_n^2+e^{-2bx_{n}}g^{\partial M}$   near the boundary
  and studied the spectrum of the Laplace-Beltrami operator for p-forms.
 In \cite{FB}, the authors dealed with a  particular class of warped products, i.e. when the pseudo-
metric in the base is affected by a conformal change. Motivated by Antoci and Dobarro etc., in this paper we consider the following metric
 near the boundary
$ g^M=\frac{1}{\varphi(x_n)}g^{\partial M}+\psi(x_n)dx_n^2$, which we call the conformal Robertson-Walker metric.
We  derive the gravitational action on boundary by the
noncommutative residue  associated with Dirac operator for the above metric.  For lower
dimensional manifolds with boundary, we compute  $\widetilde{{\rm Wres}}[\pi^+D^{-p_{1}}\circ\pi^+D^{-p_{2}}]$ with
 the conformal Robertson-Walker metric $g^M$ on $M$ , and we  get a generalized Kastler-Kalau-Walze theorem  for  lower dimensional spin manifolds.
For 6-dimensional manifolds with boundary with the conformal Robertson-Walker metric,  we  get
$\widetilde{{\rm Wres}}[\pi^+D^{-1}\circ\pi^+D^{-3}]$ is proportional to the Einstein-Hilbert action for manifolds with boundary. Which
gives a operator theoretic explanations of the total gravitational action for manifolds with boundary,
i.e, $I_{\rm {Gr}}=\frac{-3}{80 \pi\Omega_5 }\widetilde{{\rm Wres}}[\pi^+D^{-1}\circ\pi^+D^{-3}]$.

 This paper is organized as follows: In Section 2, we define lower dimensional volumes of spin manifolds with boundary. In
Section 3, for $4$-dimensional spin manifolds with boundary of  the conformal Robertson-Walker metric and the associated Dirac operator $D$ ,
 we compute the lower dimensional volume ${\rm Vol}^{(1,1)}_4$ and get a Kastler-Kalau-Walze type theorem in this case. In
Section 4, for $6$-dimensional spin manifolds with boundary of the conformal Robertson-Walker metric and the associated Dirac operator $D$ ,
 we compute the lower dimensional volume ${\rm Vol}^{(2,2)}_6$ and get a Kastler-Kalau-Walze type theorem in this case.
 In Section 5, for $6$-dimensional spin manifolds with boundary of the conformal Robertson-Walker metric and the associated Dirac operator $D$ and
 $D^{3}$ ,  We compute the lower dimensional volume ${\rm Vol}^{(1,3)}_6$ for $6$-dimensional spin manifolds with
boundary and obtain the
 noncommutative residue of the composition of $\pi^+D^{-1}$ and $\pi^+D^{-3}$ is proportional to the Einstein-Hilbert action for
 manifolds with boundary..

\section{Lower dimensional volumes of spin manifolds with boundary}
 In order to define lower dimensional volumes of spin manifolds with boundary, we need some basic facts and formulae about Boutet de
Monvel's calculus and the definition of the noncommutative residue for manifolds with boundary. We can find them in Section 2,3 \cite{Wa2}
and Section 2.1 \cite{Wa3}.

Let $M$ be a $n$-dimensional compact oriented spin manifold with boundary $\partial M$. We assume that the metric $g^M$ on $M$
has the following form near the boundary,
 \begin{equation}
 g^M=\frac{1}{\varphi(x_n)}g^{\partial M}+\psi(x_n)dx_n^2,
\end{equation}
where $g^{\partial M}$ is the metric on  ${\partial M}$, $\varphi(x_n), \psi(x_n)>0$ and  $\varphi(0)=\psi(0)=1$.
Let $D$ be the Dirac operator associated to $g$ on the spinors bundle $S(TM)$\cite{Wa3}.
 Let $p_1,p_2$ be nonnegative integers and $p_1+p_2\leq n$. From Section 2 in \cite{Wa3}, we have

\begin{defn}
Lower dimensional volumes of spin manifolds with boundary are defined by
 \begin{equation}
{\rm Vol}^{(p_1,p_2)}_nM:= \widetilde{{\rm Wres}}[\pi^+D^{-p_1}\circ\pi^+D^{-p_2}].
\end{equation}
\end{defn}
 Denote by $\sigma_l(A)$ the $l$-order symbol of an operator $A$. By (2.1.4)-(2.1.8) in \cite{Wa3}, we get
\begin{equation}
\widetilde{{\rm Wres}}[\pi^+D^{-p_1}\circ\pi^+D^{-p_2}]=\int_M\int_{|\xi|=1}{\rm
trace}_{S(TM)}[\sigma_{-n}(D^{-p_1-p_2})]\sigma(\xi)dx+\int_{\partial M}\Phi,
\end{equation}
and
\begin{eqnarray}
\Phi &=&\int_{|\xi'|=1}\int^{+\infty}_{-\infty}\sum^{\infty}_{j, k=0}\sum\frac{(-i)^{|\alpha|+j+k+1}}{\alpha!(j+k+1)!}
\times {\rm trace}_{S(TM)}[\partial^j_{x_n}\partial^\alpha_{\xi'}\partial^k_{\xi_n}\sigma^+_{r}(D^{-p_1})(x',0,\xi',\xi_n)
\nonumber\\
&&\times\partial^\alpha_{x'}\partial^{j+1}_{\xi_n}\partial^k_{x_n}\sigma_{l}(D^{-p_2})(x',0,\xi',\xi_n)]d\xi_n\sigma(\xi')dx',
\end{eqnarray}
 where the sum is taken over $r-k-|\alpha|+l-j-1=-n,~~r\leq -p_1,l\leq -p_2$.

\section{A Kastler-Kalau-Walze type theorem for $4$-dimensional spin manifolds with boundary of  conformal warped product metric }
 In this section, We compute the lower dimensional volume ${\rm Vol}^{(1,1)}_4$ for $4$-dimensional spin manifolds with
boundary of  conformal  warped product metric  and get a Kastler-Kalau-Walze type theorem in this case.

Since $[\sigma_{-4}(D^{-2})]|_M$
has the same expression as $\sigma_{-4}(D^{-2})$ in the case of
manifolds without boundary in \cite{Ka}, \cite{KW}, \cite{Ac} and \cite{Wa3}, we have
\begin{equation}
\int_M\int_{|\xi|=1}{\rm tr}[\sigma_{-4}(D^{-2})]\sigma(\xi)dx=-\frac{\Omega_4}{3}\int_Ms{\rm dvol}_M,
\end{equation}
where $\Omega_n=\frac{2\pi^{\frac{n}{2}}}{\Gamma(\frac{n}{2})}$. So we only need to compute $\int_{\partial M}\Phi$.

Firstly, we compute the symbol $\sigma(D^{-1})$ of $D^{-1}$. Recall the definition of the Dirac operator $D$  \cite{FGLS}\cite{Y}.
 Let $\nabla^L$ denote the Levi-Civita connection about $g^M$. In the local coordinates $\{x_i; 1\leq i\leq n\}$ and the
 fixed orthonormal frame $\{\widetilde{e_1},\cdots,\widetilde{e_n}\}$, the connection matrix $(\omega_{s,t})$ is defined by
\begin{equation}
\nabla^L(\widetilde{e_1},\cdots,\widetilde{e_n})= (\widetilde{e_1},\cdots,\widetilde{e_n})(\omega_{s,t}).
\end{equation}

 The Dirac operator is defined by
\begin{equation}
D=\sum^n_{i=1}c(\widetilde{e_i})\Big[\widetilde{e_i}
-\frac{1}{4}\sum_{s,t}\omega_{s,t}(\widetilde{e_i})c(\widetilde{e_s})c(\widetilde{e_t})\Big].
\end{equation}
where $c(\widetilde{e_i})$ denotes the Clifford action.

Then,
\begin{equation}
\sigma_1(D)=\sqrt{-1}c(\xi);~~~~ \sigma_0(D)
=-\frac{1}{4}\sum_{i,s,t}\omega_{s,t}(\widetilde{e_i})c(\widetilde{e_i})c(\widetilde{e_s})c(\widetilde{e_t}),
\end{equation}
where $\xi=\sum^n_{i=1}\xi_idx_i$ denotes the cotangent vector.

By Lemma 2.1 in \cite{Wa3} , we have

\begin{lem}
\begin{eqnarray}
\sigma_{-1}(D^{-1})&=&\frac{\sqrt{-1}c(\xi)}{|\xi|^2}; \\
\sigma_{-2}(D^{-1})&=&\frac{c(\xi)\sigma_{0}(D)c(\xi)}{|\xi|^4}+\frac{c(\xi)}{|\xi|^6}\sum_jc(dx_j)
\Big(\partial_{x_j}[c(\xi)]|\xi|^2-c(\xi)\partial_{x_j}(|\xi|^2)\Big) ,
\end{eqnarray}
where $\sigma_{0}(D)=-\frac{1}{4}\sum_{s,t}\omega_{s,t}(\widetilde{e_i})c(\widetilde{e_i})c(\widetilde{e_s})c(\widetilde{e_t})$.
\end{lem}

 Since $\Phi$ is a global form on $\partial M$, so for any fixed point $x_0\in\partial M$, we can choose the normal coordinates
$U$ of $x_0$ in $\partial M$ (not in $M$) and compute $\Phi(x_0)$ in the coordinates $\widetilde{U}=U\times [0,1)\subset M$ and the
metric $ g^M=\frac{1}{\varphi(x_n)}g^{\partial M}+\psi(x_n)dx_n^2$. The dual metric of $g^M$ on $\widetilde{U}$ is
$\varphi(x_n) g^{\partial M}+\frac{1}{\psi(x_n)}dx_n^2.$
Write $g^M_{ij}=g^M(\frac{\partial}{\partial x_i},\frac{\partial}{\partial x_j});~ g_M^{ij}=g^M(dx_i,dx_j)$, then

\begin{equation}
[g^M_{i,j}]= \Big[\begin{array}{lcr}
  \frac{1}{\varphi(x_n)}[g_{i,j}^{\partial M}]  & 0  \\
   0  &  \psi(x_n)
\end{array}\Big];~~~
[g_M^{i,j}]= \Big[\begin{array}{lcr}
  \varphi(x_n)[g^{i,j}_{\partial M}]  & 0  \\
   0  &  \frac{1}{\psi(x_n)}
\end{array}\Big],
\end{equation}
and
 \begin{equation}
\partial_{x_s}g_{ij}^{\partial M}(x_0)=0, 1\leq i,j\leq n-1; ~~~g_{ij}^M(x_0)=\delta_{ij}.
\end{equation}

Let $n=4$ and $\{e_1,\cdots,e_{n-1}\}$ be an orthonormal frame field in $U$ about $g^{\partial M}$ which is parallel along geodesics and
$e_i(x_0)=\frac{\partial}{\partial x_i}(x_0)$, then $\{\widetilde{e_1}=\sqrt{\varphi(x_n)}e_1,\cdots,
\widetilde{e_{n-1}}=\sqrt{\varphi(x_n)}e_{n-1},
\widetilde{e_n}=\frac{1}{\sqrt{\psi(x_n)}}dx_n\}$ is the orthonormal frame field in $\widetilde U$ about $g^M$. Locally $S(TM)|_{\widetilde {U}}\cong
\widetilde {U}\times\wedge^* _{\bf C}(\frac{n}{2}).$
Let $\{f_1,\cdots,f_4\}$ be the orthonormal basis of $\wedge^* _{\bf C}(\frac{n}{2})$.
Take a spin frame field $\sigma:~\widetilde {U}\rightarrow {\rm Spin}(M)$ such that $\pi\sigma=
\{\widetilde{e_1},\cdots,\widetilde{e_n}\}$, where $\pi :~{\rm Spin}(M)\rightarrow O(M)$ is a double covering,
 then $\{[(\sigma,f_i)],~1\leq i\leq 4\}$ is an orthonormal frame of $S(TM)|_{\widetilde {U}}.$ In the following, since the global form
$\Phi$ is independent of the choice of the local frame, so we can compute ${\rm tr}_{S(TM)}$ in the frame $\{[(\sigma,f_i)],~1\leq
i\leq 4\}.$ Let $\{E_1,\cdots,E_n\}$ be the canonical basis of
${\bf R}^n$ and $c(E_i)\in {\rm cl}_{\bf C}(n)\cong {\rm Hom}(\wedge^*_{\bf C}(\frac{n}{2}),\wedge^* _{\bf C}(\frac{n}{2}))$
be the Clifford action. By \cite{Wa3}, \cite{Wa4} and \cite{Y} , we have
\begin{equation}
c( \widetilde{e_i})=[(\sigma,c(E_i))];~ c( \widetilde{e_i})[(\sigma,f_i)]=[(\sigma,c(E_i)f_i)];~
\frac{\partial}{\partial x_i}=[(\sigma,\frac{\partial}{\partial
x_i})],
\end{equation}
then we have $\frac{\partial}{\partial x_i}c( \widetilde{e_i})=0$ in the above frame.
Therefore, we obtain

\begin{lem} For $n$-dimensional spin manifolds with boundary,
\begin{eqnarray}
\partial_{x_j}(|\xi|_{g^M}^2)(x_0)&=&\left\{
       \begin{array}{c}
        0,  ~~~~~~~~~~ ~~~~~~~~~~ ~~~~~~~~~~~~~~~~~~~~{\rm if }~j<n; \\[2pt]
        \varphi'(0)|\xi'|^2_{g^{\partial M}}-\psi'(0)\xi_n^{2}, ~~~~~~~~~~~~~~~{\rm if }~j=n.
       \end{array}
    \right. \\
\partial_{x_j}[c(\xi)](x_0)&=&\left\{
       \begin{array}{c}
      0,  ~~~~~~~~~~ ~~~~~~~~~~ ~~~~~~~~~~~~~~~~~~~~{\rm if }~j<n;\\[2pt]
   \partial_{x_n}[c(\xi')](x_0)
  +\xi_n\partial_{x_n}[c(dx_{n}](x_0) , ~~{\rm if }~j=n.
       \end{array}
    \right.
\end{eqnarray}
 where $\xi=\xi'+\xi_ndx_n$.
\end{lem}
\begin{proof}
By the equality
  $\partial_{x_j}(|\xi|_{g^M}^2)(x_0)=\partial_{x_j}\big(\varphi(x_n)g^{l,m}_{\partial M}(x')\xi_l\xi_m+\psi(x_n)\xi_n^2\big)(x_0)$ and
(3.7), then (3.9) is correct. By Lemma A.1 in \cite{Wa3}, (3.10) is correct.
\end{proof}
 In order to compute $\sigma_{0}(D)(x_0)$, we need to compute $\omega_{s,t}(\widetilde{e_i})(x_0)$.
 By Appendix in \cite{Wa3}, we have

\begin{lem}  For $n$-dimensional spin manifolds with boundary,
When $i<n$ ,~$\omega_{n,i}(\widetilde{e_i})(x_0)=\frac{1}{2} \varphi'(0);$ and $\omega_{i,n}(\widetilde{e_i})(x_0)=-\frac{1}{2} \varphi'(0).$
 In other cases, $\omega_{s,t}(\widetilde{e_i})(x_0)=0$.
\end{lem}
Combining (3.3) and Lemma 3.3, we obtain

\begin{lem} For $4$-dimensional spin manifolds with boundary,
 \begin{equation}
\sigma_{0}(D)(x_0)=-\frac{3}{4} \varphi'(0)c(dx_n).
\end{equation}
\end{lem}

 Now we can compute $\Phi$ (see formula (2.4) for the definition of $\Phi$), since the sum is taken over $
-r-l+k+j+|\alpha|=3,~~r,~l\leq -1,$ then we have the following five cases:

{\bf case a)~I)}~$r=-1,~l=-1,~k=j=0,~|\alpha|=1$

From (2.4) we have
 \begin{equation}
{\rm case~a)~I)}=-\int_{|\xi'|=1}\int^{+\infty}_{-\infty}\sum_{|\alpha|=1}
 {\rm trace}[\partial^\alpha_{\xi'}\pi^+_{\xi_n}\sigma_{-1}(D^{-1})\times
 \partial^\alpha_{x'}\partial_{\xi_n}\sigma_{-1}(D^{-1})](x_0)d\xi_n\sigma(\xi')dx'.
\end{equation}
By Lemma 3.2, for $i<n$, then
\begin{eqnarray}
\partial_{x_i}\sigma_{-1}(D^{-1})(x_0)&=&\partial_{x_i}\left(\frac{\sqrt{-1}c(\xi)}{|\xi|^2}\right)(x_0) \nonumber\\
&=&\frac{\sqrt{-1}\partial_{x_i}[c(\xi)](x_0)}{|\xi|^2}
-\frac{\sqrt{-1}c(\xi)\partial_{x_i}(|\xi|^2)(x_0)}{|\xi|^4}\nonumber\\
&=&0.
\end{eqnarray}
Then case a) I) vanishes.

 {\bf case a)~II)}~$r=-1,~l=-1, ~k=|\alpha|=0,~j=1$

From (2.4) we have
 \begin{equation}
{\rm case~ a)~II)}=-\frac{1}{2}\int_{|\xi'|=1}\int^{+\infty}_{-\infty} {\rm
trace} [\partial_{x_n}\pi^+_{\xi_n}\sigma_{-1}(D^{-1})\times
\partial_{\xi_n}^2\sigma_{-1}(D^{-1})](x_0)d\xi_n\sigma(\xi')dx'.
\end{equation}
By Lemma 3.1 and Lemma 3.2, we have
\begin{eqnarray}
\partial^2_{\xi_n}\sigma_{-1}(D^{-1})&=&\sqrt{-1}\Big(-\frac{6\xi_nc(dx_n)+2c(\xi')}
{|\xi|^4}+\frac{8\xi_n^2c(\xi)}{|\xi|^6}\Big) \nonumber\\
&=&\frac{6i\xi_n^{2}-2i}{(1+\xi_n^{2})^3}c(\xi')+\frac{2i\xi_n^{3}-6i\xi_n}{(1+\xi_n^{2})^3}c(dx_n),
\end{eqnarray}
and
\begin{eqnarray}
\partial_{x_n}\sigma_{-1}(D^{-1})(x_0)
&=&\frac{\sqrt{-1}\partial_{x_n}[c(\xi)](x_0)}{|\xi|^2}
-\frac{\sqrt{-1}c(\xi)\partial_{x_n}(|\xi|^2)(x_0)}{|\xi|^4}\nonumber\\
&=&\frac{i\partial_{x_n}[c(\xi')](x_0)}{|\xi|^2}+\frac{i\xi_n\partial_{x_n}[c(dx_n)](x_0)}{|\xi|^2}
-\frac{\big(i\varphi'(0)-i\xi_n^{2}\psi'(0)\big)c(\xi)}{|\xi|^4}.
\end{eqnarray}
By (2.1.1) in \cite{Wa3}, (3.3) in \cite{WW} and the Cauchy integral formula, then
\begin{eqnarray}
\pi^+_{\xi_n}\left[\frac{c(\xi)}{|\xi|^4}\right](x_0)|_{|\xi'|=1}&=&\pi^+_{\xi_n}\left[\frac{c(\xi')+\xi_nc(dx_n)}{(1+\xi_n^2)^2}\right] \nonumber\\
&=&\frac{1}{2\pi i}\lim_{u\rightarrow
0^-}\int_{\Gamma^+}\frac{\frac{c(\xi')+\eta_nc(dx_n)}{(\eta_n+i)^2(\xi_n+iu-\eta_n)}}
{(\eta_n-i)^2}d\eta_n   \nonumber\\
&=&\left[\frac{c(\xi')+\eta_nc(dx_n)}{(\eta_n+i)^2(\xi_n-\eta_n)}\right]^{(1)}\mid_{\eta_n=i} \nonumber\\
&=&\frac{-2-i\xi_n}{4(\xi_n-i)^{2}}c(\xi')-\frac{i}{4(\xi_n-i)^2}c(dx_n).
\end{eqnarray}
Similarly,
\begin{equation}
\pi^+_{\xi_n}\left[\frac{i\partial_{x_n}[c(\xi')]}{|\xi|^2}\right](x_0)|_{|\xi'|=1}
=\frac{\partial_{x_n}[c(\xi')](x_0)}{2(\xi_n-i)},
\end{equation}
\begin{equation}
\pi^+_{\xi_n}\left[\frac{i\xi_n\partial_{x_n}[c(dx_n)]}{|\xi|^2}\right](x_0)|_{|\xi'|=1}
=\frac{i\partial_{x_n}[c(dx_n)](x_0)}{2(\xi_n-i)},
\end{equation}
and
\begin{equation}
\pi^+_{\xi_n}\left[\frac{\xi_n^{2}c(\xi)}{|\xi|^4}\right](x_0)|_{|\xi'|=1}
=\frac{-i\xi_n}{4(\xi_n-i)^{2}}c(\xi')+\frac{2\xi_n-i}{4(\xi_n-i)^2}c(dx_n).
\end{equation}
Combining (3.17)-(3.21), we obtain
\begin{eqnarray}
\partial_{x_n}\pi^+_{\xi_n}\sigma_{-1}(D^{-1})(x_0)|_{|\xi'|=1}
&=&\frac{\partial_{x_n}[c(\xi')](x_0)}{2(\xi_n-i)}+\frac{i\partial_{x_n}[c(dx_n)](x_0)}{2(\xi_n-i)}\nonumber\\
&&+\frac{(2i-\xi_n)\varphi'(0)+\xi_n\psi'(0)}{4(\xi_n-i)^{2}}c(\xi')\nonumber\\
&&+\frac{-\varphi'(0)+(1+2i\xi_n)\psi'(0)}{4(\xi_n-i)^2}c(dx_n).
\end{eqnarray}
Since $n=4$, ${\rm tr}_{S(TM)}[{\rm id}]={\rm dim}(\wedge^*(2))=4$.
By the relation of the Clifford action and ${\rm tr}{AB}={\rm tr }{BA}$, we have the equalities:
\begin{eqnarray}
&&{\rm tr}[c(\xi')c(dx_n)]=0;~~{\rm tr}[c(dx_n)^2]=-4;~~{\rm tr}[c(\xi')^2](x_0)|_{|\xi'|=1}=-4; \nonumber\\
&&{\rm tr}[\partial_{x_n}[c(\xi')]c(dx_n)]=0;~~{\rm tr}[\partial_{x_n}[c(\xi')]c(\xi')](x_0)|_{|\xi'|=1}=-2\varphi'(0);\nonumber\\
&&{\rm tr}[\partial_{x_n}[c(dx_n)]c(\xi')]=0;~~{\rm tr}[\partial_{x_n}[c(dx_n)]c(dx_n)](x_0)|_{|\xi'|=1}=2\psi'(0).
\end{eqnarray}
From (3.16), (3.22)and (3.23), we have
\begin{eqnarray}
&&{\rm tr}\Big\{\Big[\frac{\partial_{x_n}[c(\xi')](x_0)}{2(\xi_n-i)}+\frac{i\partial_{x_n}[c(dx_n)](x_0)}{2(\xi_n-i)}
+\frac{(2i-\xi_n)\varphi'(0)+\xi_n\psi'(0)}{4(\xi_n-i)^{2}}c(\xi')
+\frac{-\varphi'(0)+(1+2i\xi_n)\psi'(0)}{4(\xi_n-i)^2}c(dx_n)\Big]\nonumber\\
&&~~~~\times
\Big[\frac{6i\xi_n^{2}-2i}{(1+\xi_n^{2})^3}c(\xi')+\frac{2i\xi_n^{3}-6i\xi_n}{(1+\xi_n^{2})^3}c(dx_n)\Big]
\Big\}(x_0)|_{|\xi'|=1} \nonumber\\
&=&\frac{2\big(i\varphi'(0)+\xi_n\psi'(0)\big)}{(\xi_n-i)^2(\xi_n+i)^3}.
\end{eqnarray}
Substituting (3.24) into (3.15), we have
\begin{eqnarray}
{\rm case~ a)~II)}&=&-\int_{|\xi'|=1}\int^{+\infty}_{-\infty}
 \frac{i\varphi'(0)+\xi_n\psi'(0)}{(\xi_n-i)^2(\xi_n+i)^3}d\xi_n\sigma(\xi')dx' \nonumber\\
&=&-\Omega_3\int_{\Gamma^+} \frac{i\varphi'(0)+\xi_n\psi'(0)}{(\xi_n-i)^2(\xi_n+i)^3}d\xi_ndx'\nonumber\\
&=&-\Omega_3 2\pi i\Big[\frac{i\varphi'(0)+\xi_n\psi'(0)}{(\xi_n+i)^3}\Big]^{(1)}|_{\xi_n=i}dx'\nonumber\\
&=&-\frac{1}{8}\big(3\varphi'(0)+\psi'(0)\big)\pi \Omega_3dx',
\end{eqnarray}
 where $\Omega_4$ is the canonical volume of $S^4$.

{\bf case a)~III)}~$r=-1,~l=-1,~j=|\alpha|=0,~k=1$\\

From (2.4) we have
 \begin{equation}
{\rm case~ a)~III)}=-\frac{1}{2}\int_{|\xi'|=1}\int^{+\infty}_{-\infty}
{\rm trace} [\partial_{\xi_n}\pi^+_{\xi_n}\sigma_{-1}(D^{-1})\times
\partial_{\xi_n}\partial_{x_n}\sigma_{-1}(D^{-1})](x_0)d\xi_n\sigma(\xi')dx'.
\end{equation}
By (2.2.29) in \cite{Wa3}, we have
 \begin{equation}
\partial_{\xi_n}\pi^+_{\xi_n}\sigma_{-1}(D^{-1})(x_0)|_{|\xi'|=1}=\frac{-c(\xi')}{2(\xi_n-i)^2}+\frac{-ic(dx_n)}{2(\xi_n-i)^2}.
\end{equation}
From (3.5) we have
\begin{eqnarray}
\partial_{\xi_n}\partial_{x_n}\sigma_{-1}(D^{-1})(x_0)|_{|\xi'|=1}
&=&\frac{-2i\xi_n}{(1+\xi_n^{2})^2}\partial_{x_n}[c(\xi')](x_0)+\frac{i-i\xi_n^{2}}{(1+\xi_n^{2})^2}\partial_{x_n}[c(dx_n)](x_0)\nonumber\\
&&+\frac{(3i\xi_n^{2}-i)\varphi'(0)+(3i\xi_n^{2}-i\xi_n^{4})\psi'(0)}{(1+\xi_n^{2})^3}c(dx_n)\nonumber\\
&&+\frac{4i\xi_n\varphi'(0)+(2i\xi_n-2i\xi_n^{3})\psi'(0)}{(1+\xi_n^{2})^3}c(\xi').
\end{eqnarray}
Combining (3.27) and (3.28), we obtain
\begin{eqnarray}
&&{\rm  tr}\Big\{\Big[\frac{-c(\xi')-ic(dx_n)}{2(\xi_n-i)^2}\Big]\times
\Big[\frac{-2i\xi_n}{(1+\xi_n^{2})^2}\partial_{x_n}[c(\xi')](x_0)+\frac{i-i\xi_n^{2}}{(1+\xi_n^{2})^2}\partial_{x_n}[c(dx_n)](x_0)\nonumber\\
&&~~~+\frac{(3i\xi_n^{2}-i)\varphi'(0)+(3i\xi_n^{2}-i\xi_n^{4})\psi'(0)}{(1+\xi_n^{2})^3}c(dx_n)
+\frac{4i\xi_n\varphi'(0)+(2i\xi_n-2i\xi_n^{3})\psi'(0)}{(1+\xi_n^{2})^3}c(\xi')
\Big]
\Big\}(x_0)|_{|\xi'|=1} \nonumber\\
&=&\frac{-2i\varphi'(0)+(\xi_n-i)\psi'(0)}{(\xi_n-i)^2(\xi_n+i)^3}.
\end{eqnarray}
Substituting (3.29) into (3.26), one sees that
\begin{eqnarray}
{\rm {\bf case~a)~III)}}&=&
-\frac{1}{2}\int_{|\xi'|=1}\int^{+\infty}_{-\infty}\frac{-2i\varphi'(0)+(\xi_n-i)\psi'(0)}{(\xi_n-i)^2(\xi_n+i)^3}d\xi_n\sigma(\xi')dx' \nonumber\\
&=& -\frac{1}{2}\times2 \pi i\Big[\frac{-2i\varphi'(0)+(\xi_n-i)\psi'(0)}{(\xi_n+i)^3} \Big]^{(1)}|_{\xi_n=i}\Omega_3dx'\nonumber\\
&=&\frac{1}{8}\big(3\varphi'(0)+\psi'(0)\big)\pi \Omega_3dx'.
\end{eqnarray}

{\bf case b)}~$r=-2,~l=-1,~k=j=|\alpha|=0$\\

From (2.4) we have
 \begin{equation}
{\rm case~ b)}=-i\int_{|\xi'|=1}\int^{+\infty}_{-\infty}{\rm trace} [\pi^+_{\xi_n}\sigma_{-2}(D^{-1})\times
\partial_{\xi_n}\sigma_{-1}(D^{-1})](x_0)d\xi_n\sigma(\xi')dx'.
\end{equation}
By Lemma 3.1 and Lemma 3.2, we have
\begin{equation}
\partial_{\xi_n}\sigma_{-1}(D^{-1})=\frac{-2i\xi_n}{(1+\xi_n^{2})^2}c(\xi')+\frac{i-i\xi_n^{2}}{(1+\xi_n^{2})^2}c(dx_n),
\end{equation}
and
\begin{eqnarray}
\sigma_{-2}(D^{-1})(x_0)&=& \frac{c(\xi)\sigma_{0}(D)(x_0)c(\xi)}{|\xi|^4}
+\frac{1}{|\xi|^4}c(\xi)c(dx_n)\Big[\partial_{x_n}[c(\xi')](x_0)+\xi_n\partial_{x_n}[c(dx_n)](x_0)\Big]  \nonumber\\
&&-\frac{1}{|\xi|^6}c(\xi)c(dx_n)c(\xi)\Big[\varphi'(0)-\xi_n^{2}\psi'(0)\Big].
\end{eqnarray}
Then
\begin{eqnarray}
\pi^+_{\xi_n}\sigma_{-2}(D^{-1})(x_0)|_{|\xi'|=1}&=& \pi^+_{\xi_n}\Big[\frac{c(\xi)\sigma_{0}(D)(x_0)c(\xi)}{|\xi|^4}\Big]\nonumber\\
&&+\pi^+_{\xi_n}\Big[\frac{1}{|\xi|^4}c(\xi)c(dx_n)\big[\partial_{x_n}[c(\xi')](x_0)+\xi_n\partial_{x_n}[c(dx_n)](x_0)\big]\Big]  \nonumber\\
&&+\pi^+_{\xi_n}\Big[\frac{1}{|\xi|^6}c(\xi)c(dx_n)c(\xi)\big[\xi_n^{2}\psi'(0)-\varphi'(0)\big]\Big] \nonumber\\
&:=&A+B+C.
\end{eqnarray}
 Similarly to (3.18), by Lemma 3.1 we have
 \begin{equation}
A=\frac{-1}{4(\xi_n-i)^{2}}\Big[\frac{-3\varphi'(0)}{4}(2+i\xi_n)c(\xi')c(dx_n)c(\xi')-\frac{-3i\xi_n\varphi'(0)}{4} c(dx_n)
+\frac{3i\varphi'(0)}{2} c(\xi')\Big].
\end{equation}
And
\begin{eqnarray}
B&=&\frac{-1}{4(\xi_n-i)^{2}}\Big[(2+i\xi_n)c(\xi')c(dx_n)\partial_{x_n}[c(\xi')](x_0)
+ic(\xi')c(dx_n)\partial_{x_n}[c(dx_n)](x_0) \nonumber\\
&&-i\partial_{x_n}[c(\xi')](x_0)
-i\xi_n\partial_{x_n}[c(dx_n)](x_0) \Big],
\end{eqnarray}
\begin{eqnarray}
C&=&\frac{\xi_n}{16(\xi_n-i)^{3}}\Big[-3\varphi'(0)-i\xi_n\varphi'(0)+\psi'(0)+3i\xi_n\psi'(0)  \Big]c(dx_n) \nonumber\\
  &&+\frac{1}{8(\xi_n-i)^{3}}\Big[-3\varphi'(0)-i\xi_n\varphi'(0)+\psi'(0)+3i\xi_n\psi'(0)  \Big]c(\xi')\nonumber\\
  &&+\frac{1}{16(\xi_n-i)^{3}}\Big[-8i\varphi'(0)+9\xi_n \varphi'(0)+3i\xi_n^{2}\varphi'(0)-3\xi_n\psi'(0)-i\xi_n^{2}\psi'(0)
        \Big]c(\xi')c(dx_n)c(\xi').
\end{eqnarray}
 By (3.32) and (3.35)-(3.37), we obtain
\begin{eqnarray}
{\rm tr }[A\times\partial_{\xi_n}q_{-1}(x_0)]|_{|\xi'|=1}&=&\frac{3i\varphi'(0)}{2(1+\xi_n^{2})^{2}};\\
{\rm tr }[B\times\partial_{\xi_n}q_{-1}(x_0)]|_{|\xi'|=1}&=&
 \frac{-2\varphi'(0)-i\xi_n\varphi'(0)+\xi_n^{2}\varphi'(0)+i\xi_n\psi'(0)+\xi_n^{2}\psi'(0)}{2(\xi_n-i)^{3}(\xi_n+i)^{2}};
\\
{\rm tr }[C\times\partial_{\xi_n}q_{-1}(x_0)]|_{|\xi'|=1}&=&
 \frac{4\varphi'(0)+i\xi_n\varphi'(0)-\xi_n^{2}\varphi'(0)-3i\xi_n\psi'(0)-\xi_n^{2}\psi'(0)}{2(\xi_n-i)^{3}(\xi_n+i)^{2}}.
\end{eqnarray}
 Combining (3.31), (3.38), (3.39) and (3.40), we obtain
\begin{eqnarray}
{\bf case~ b)}&=&
 -i \int_{|\xi'|=1}\int^{+\infty}_{-\infty}\frac{-5i\varphi'(0)+3\xi_n\varphi'(0)-2\xi_n\psi'(0)}{-2i(\xi_n-i)^3(\xi_n+i)^2}
  d\xi_n\sigma(\xi')dx' \nonumber\\
 &=& -i \Omega_3  \int_{\Gamma^+} \frac{-5i\varphi'(0)+3\xi_n\varphi'(0)-2\xi_n\psi'(0)}{-2i(\xi_n-i)^3(\xi_n+i)^2}d\xi_ndx'\nonumber\\
&=& -i \Omega_3 \frac{2 \pi i}{2!}\Big[\frac{-5i\varphi'(0)+3\xi_n\varphi'(0)-2\xi_n\psi'(0)}{-2i(\xi_n+i)^2}
     \Big]^{(2)}|_{\xi_n=i}dx'\nonumber\\
&=&\frac{1}{8}\big(9\varphi'(0)-\psi'(0)\big)\pi \Omega_3dx'.
\end{eqnarray}

 {\bf  case c)}~$r=-1,~l=-2,~k=j=|\alpha|=0$\\

From (2.4) we have
 \begin{equation}
{\rm case~ c)}=-i\int_{|\xi'|=1}\int^{+\infty}_{-\infty}{\rm trace} [\pi^+_{\xi_n}\sigma_{-1}(D^{-1})\times
\partial_{\xi_n}\sigma_{-2}(D^{-1})](x_0)d\xi_n\sigma(\xi')dx'.
\end{equation}
By (2.2.44) in \cite{Wa3}, we have
 \begin{equation}
\pi^+_{\xi_n}\sigma_{-1}(D^{-1})(x_0)|_{|\xi'|=1}=\frac{c(\xi')+ic(dx_n)}{2(\xi_n-i)};
\end{equation}
By (3.33) we have
\begin{eqnarray}
\partial_{\xi_n}\sigma_{-2}(D^{-1})(x_0)|_{|\xi'|=1}&=&\frac{1}{(1+\xi_n^{2})^{3}}\Big[ (3\xi_n^{2}-1)\partial_{x_n}[c(\xi')](x_0)
   +(2\xi_n^{3}-2\xi_n)\partial_{x_n}[c(dx_n)](x_0)\nonumber\\
   && +(1-3\xi_n^{2})c(\xi')c(dx_n)\partial_{x_n}[c(dx_n)](x_0)
   -4\xi_n c(\xi')c(dx_n)\partial_{x_n}[c(\xi')](x_0)\Big]\nonumber\\
  &&+\frac{1}{(1+\xi_n^{2})^{4}}\Big[\xi_n\Big(9\varphi'(0)+3\xi_n^{2}\varphi'(0)+2\psi'(0)-4\xi_n^{2}\psi'(0)
  \Big)  c(\xi')c(dx_n)c(\xi')\nonumber\\
  &&-\frac{1}{2}\Big(-7\varphi'(0)+26\xi_n^{2}\varphi'(0)+9\xi_n^{4}\varphi'(0)+12\xi_n^{2}\psi'(0)-12\xi_n^{4}\psi'(0)
  \Big) c(\xi')\nonumber\\
  &&-\frac{\xi_n}{2}\Big(-7\varphi'(0)+8\xi_n^{2}\varphi'(0)+3\xi_n^{4}\varphi'(0)+8\xi_n^{2}\psi'(0)-4\xi_n^{4}\psi'(0)
  \Big) c(dx_n)\Big] . \nonumber\\
\end{eqnarray}
Then similarly to computations of the case b), we have
\begin{eqnarray}
&&{\rm trace} [\pi^+_{\xi_n}\sigma_{-1}(D^{-1})\times \partial_{\xi_n}\sigma_{-2}(D^{-1})](x_0)|_{|\xi'|=1}\nonumber\\
&=&\frac{1}{-i(\xi_n-i)(\xi_n+i)^{4}}
\Big[6i\varphi'(0)+3\xi_n\varphi'(0)+i\psi'(0)-2\xi_n\psi'(0)\Big]
\end{eqnarray}
Combining (3.42) and (3.45), we obtain
  \begin{equation}
{\bf case~ c)}=-\frac{1}{8}\big(9\varphi'(0)-\psi'(0)\big)\pi \Omega_3dx'.
\end{equation}

Since $\Phi$ is the sum of the cases a), b) and c), so $\Phi=0$. Therefore

\begin{thm}
Let  $M$ be a $4$-dimensional
compact spin manifold with the boundary $\partial M$ and the metric
$g^M$ as above and $D$ be the Dirac operator on $\widehat{M}$, then
\begin{equation}
{\rm Vol}^{(1,1)}_4=\widetilde{{\rm Wres}}[\pi^+D^{-1}\circ\pi^+D^{-1}]=-\frac{\Omega_4}{3}\int_Ms{\rm dvol}_M.
\end{equation}
\end{thm}

Now, we recall the Einstein-Hilbert action for manifolds with boundary in \cite{Wa3},
 \begin{equation}
I_{\rm Gr}=\frac{1}{16\pi}\int_Ms{\rm dvol}_M+2\int_{\partial M}K{\rm dvol}_{\partial_M}:=I_{\rm {Gr,i}}+I_{\rm {Gr,b}},
\end{equation}
 where
 \begin{equation}
K=\sum_{1\leq i,j\leq {n-1}}K_{i,j}g_{\partial M}^{i,j};~~K_{i,j}=-\Gamma^n_{i,j},
\end{equation}
 and $K_{i,j}$ is the second fundamental form, or extrinsic
curvature. Take the metric in Section 2, then by Lemma A.2 in \cite{Wa3},
$K_{i,j}(x_0)=-\Gamma^n_{i,j}(x_0)=-\frac{1}{2}\varphi'(0),$ when $i=j<n$,
otherwise is zero. For $n=4$, we obtain
\begin{equation}
K(x_0)=\sum_{i,j}K_{i.j}(x_0)g_{\partial M}^{i,j}(x_0)=\sum_{i=1}^3K_{i,i}(x_0)=-\frac{3}{2}\varphi'(0).
\end{equation}
 So
 \begin{equation}
I_{\rm {Gr,b}}=-3\varphi'(0){\rm Vol}_{\partial M}.
\end{equation}
 Let $M$ be a $4$-dimensional manifold with boundary and $P,P'$ be two pseudodifferential operators with transmission
property (see \cite{Wa1}) on $\widehat M$. From (4.4) in \cite{Wa3},  we have
\begin{equation}
\pi^+P\circ\pi^+P'=\pi^+(PP')+L(P,P')
\end{equation}
and $L(P,P')$ is leftover term which represents the difference between
the composition $\pi^+P\circ\pi^+P'$ in Boutet de Monvel algebra and
the composition $PP'$ in the classical pseudodifferential operators
algebra. By (2.4), we define locally
\begin{eqnarray}
&&{\rm res}_{1,1}(P,P'):=-\frac{1}{2}\int_{|\xi'|=1}\int^{+\infty}_{-\infty}
{\rm trace} [\partial_{x_n}\pi^+_{\xi_n}\sigma_{-1}(P)\times
\partial_{\xi_n}^2\sigma_{-1}(P')]d\xi_n\sigma(\xi')dx'; \\
&&{\rm res}_{2,1}(P,P'):=-i\int_{|\xi'|=1}\int^{+\infty}_{-\infty}
{\rm trace} [\pi^+_{\xi_n}\sigma_{-2}(P)\times
\partial_{\xi_n}\sigma_{-1}(P')]d\xi_n\sigma(\xi')dx'.
\end{eqnarray}

Hence, they represent the difference between the composition
$\pi^+P\circ\pi^+P'$ in Boutet de Monvel algebra and the composition
$PP'$ in the classical pseudodifferential operators algebra
partially. Then
\begin{equation}
{\rm case~ a)~ II)}={\rm res}_{1,1}(D^{-1},D^{-1});~{\rm case~ b)}={\rm res}_{2,1}(D^{-1},D^{-1}).
\end{equation}

 Now, we assume $\partial M$ is flat , then
$\{dx_i=e_i\},~g^{\partial M}_{i,j}=\delta_{i,j},~\partial
_{x_s}g^{\partial M}_{i,j}=0$. So ${\rm res}_{1,1}(D^{-1},D^{-1})$
and ${\rm res}_{2,1}(D^{-1},D^{-1})$ are two global forms locally
defined by the aboved oriented orthonormal basis $\{dx_i\}$.  Let $\psi'(0)=\varphi'(0)$,
 from case a) II) and case b),  then we obtain:

\begin{thm}
Let M be a 4-dimensional flat compact connected foliation with the boundary $\partial M$ and the metric $g^{M}$ as above ,
and $D$ be the Dirac operator on $\widehat{M}$, then
\begin{eqnarray}
&& \int_{\partial M}{\rm res}_{1,1}(D^{-1},D^{-1})=\frac{\pi}{6}\Omega_3I_{\rm {Gr,b}};  \\
&&\int_{\partial M}{\rm res}_{2,1}(D^{-1},D^{-1})=-\frac{\pi}{3}\Omega_3I_{\rm {Gr,b}}.
\end{eqnarray}
\end{thm}

Nextly, for $3$-dimensional spin manifolds with boundary, we compute ${\rm Vol}^{(1,1)}_3$. By Section 5 in \cite{Wa3}, we have
 \begin{equation}
\widetilde{{\rm Wres}}[\pi^+D^{-1}\circ\pi^+D^{-1}]=\int_{\partial M}\Phi.
\end{equation}
 By (2.4), when $n=3$, we have  $ r-k-|\alpha|+l-j-1=-3,~~r\leq -1,l\leq-3$, so we get $r=-1,~ l=-1,~k=|\alpha|=j=0,$ then
  \begin{equation}
\Phi=\int_{|\xi'|=1}\int^{+\infty}_{-\infty} {\rm trace}_{S(TM)}
[ \sigma^+_{-1}(D^{-1})(x',0,\xi',\xi_n)\times \partial_{\xi_n}\sigma_{-1}
(D^{-3})(x',0,\xi',\xi_n)]d\xi_n\sigma(\xi')dx'.
\end{equation}
By (2.2.44) in \cite{Wa3}, we have
  \begin{equation}
\pi^+_{\xi_n}\sigma_{-1} (D^{-1})(x_0)|_{|\xi'|=1}=\frac{c(\xi')+ic(dx_n)}{2(\xi_n-i)}.
\end{equation}
From (3.5) we have
\begin{equation}
\partial_{\xi_n}\sigma_{-1}(D^{-1})=\frac{-2i\xi_n}{(1+\xi_n^{2})^2}c(\xi')+\frac{i-i\xi_n^{2}}{(1+\xi_n^{2})^2}c(dx_n).
\end{equation}

Since $n=3$, ${\rm tr}(id)={\rm dim}(S(TM))=2$.
By the relation of the Clifford action and ${\rm tr}{AB}={\rm tr }{BA}$,  we have the equalities:
 \begin{equation}
{\rm tr}[c(\xi')c(dx_n)]=0;~~{\rm tr}[c(dx_n)^2]=-2;~~{\rm tr}[c(\xi')^2](x_0)|_{|\xi'|=1}=-2.
\end{equation}
Hence from (3.59), (3.60) and (3.61), we have
 \begin{equation}
{\rm trace} [\sigma^+_{-1}(D^{-1})\times\partial_{\xi_n}\sigma_{-1}(D^{-1})](x_0)|_{|\xi'|=1}=\frac{-1}{(\xi_n+i)^2(\xi_n-i)}.
\end{equation}
Then
 \begin{equation}
\Phi=\frac{i\pi}{2}\Omega_2{\rm vol}_{\partial M}.
\end{equation}
where ${\rm vol}_{\partial M}$ denotes the canonical volume form of ${\partial M}$. Then
\begin{thm}
Let  $M$ be a $3$-dimensional
compact spin manifold with the boundary $\partial M$ and the metric
$g^M$ as in Section 2 and $D$ be the Dirac operator on $\widehat{M}$, then
 \begin{equation}
 {\rm Vol}^{(1,1)}_3=\frac{i\pi}{2}\Omega_2{\rm vol}_{\partial M}.
\end{equation}
where ${\rm Vol}_{\partial M}$ denotes the canonical volume of ${\partial M}$.
\end{thm}
\begin{rem}
When $\psi(x_n)=1$, we get Theorem 2.5 and Theorem 5.1 in \cite{Wa3}.
\end{rem}

\section{A Kastler-Kalau-Walze type theorem for $6$-dimensional spin manifolds with boundary of  conformal warped product metric associated with $D^{2}$}
 In this section, We compute the lower dimensional volume ${\rm Vol}^{(2,2)}_6$ for $4$-dimensional spin manifolds with
boundary of  warped product metric $ g^M=\frac{1}{\varphi(x_n)}g^{\partial M}+\psi(x_n)dx_n^2$
 and get a Kastler-Kalau-Walze type theorem in this case.

Since $[\sigma_{-6}(D^{-4})]|_M$ has the same expression as $\sigma_{-6}(D^{-4})$ in the case of
manifolds without boundary in  \cite{Wa4}, we have
\begin{equation}
\int_M\int_{|\xi|=1}{\rm tr}[\sigma_{-6}(D^{-4})]\sigma(\xi)dx=-\frac{5\Omega_6}{3}\int_Ms{\rm dvol}_M,
\end{equation}
where $\Omega_n=\frac{2\pi^{\frac{n}{2}}}{\Gamma(\frac{n}{2})}$. So we only need to compute $\int_{\partial M}\Phi$.

By Lemma 1 in \cite{Wa4} , we have

\begin{lem}
\begin{eqnarray}
&&\sigma_{-2}(D^{-2})=|\xi|^{-2}; \\
&&\sigma_{-3}(D^{-2})=-\sqrt{-1}|\xi|^{-4}\xi_k(\Gamma^k-2\delta^k)-\sqrt{-1}|\xi|^{-6}2\xi^j\xi_\alpha\xi_\beta
\partial_jg^{\alpha\beta}.
\end{eqnarray}
\end{lem}

 Now we can compute $\Phi$ (see formula (2.4) for the definition of $\Phi$), since the sum is taken over $
-r-l+k+j+|\alpha|=5,~~r,~l\leq -2,$ then we have the following five cases:

{\bf case a)~I)}~$r=-2,~l=-2,~k=j=0,~|\alpha|=1$

From (2.4) we have
 \begin{equation}
{\rm case~a)~I)}=-\int_{|\xi'|=1}\int^{+\infty}_{-\infty}\sum_{|\alpha|=1}
 {\rm trace}[\partial^\alpha_{\xi'}\pi^+_{\xi_n}\sigma_{-2}(D^{-2})\times
 \partial^\alpha_{x'}\partial_{\xi_n}\sigma_{-2}(D^{-2})](x_0)d\xi_n\sigma(\xi')dx'.
\end{equation}
By Lemma 3.2, for $i<n$, then
 \begin{equation}
\partial_{x_i}\sigma_{-2}(D^{-2})(x_0)=\partial_{x_i}\left(\frac{1}{|\xi|^2}\right)(x_0)
=-\frac{\partial_{x_i}(|\xi|^2)(x_0)}{|\xi|^4}\nonumber\\
=0.
\end{equation}
Then case a) I) vanishes.

 {\bf case a)~II)}~$r=-1,~l=-1, ~k=|\alpha|=0,~j=1$

From (2.4) we have
 \begin{equation}
{\rm case~ a)~II)}=-\frac{1}{2}\int_{|\xi'|=1}\int^{+\infty}_{-\infty} {\rm
trace} [\partial_{x_n}\pi^+_{\xi_n}\sigma_{-2}(D^{-2})\times
\partial_{\xi_n}^2\sigma_{-2}(D^{-2})](x_0)d\xi_n\sigma(\xi')dx'.
\end{equation}
By Lemma 3.1 and Lemma 3.2, we have
 \begin{equation}
\partial^2_{\xi_n}\sigma_{-2}(D^{-2})(x_0)=\partial^2_{\xi_n}(\frac{1}{|\xi|^2})(x_0)=\frac{-2+6\xi_n^{2}}{(1+\xi_n^{2})^3}.
\end{equation}
and
 \begin{equation}
\partial_{x_n}\sigma_{-2}(D^{-2})(x_0)=\frac{\xi_n^{2}\psi'(0)-\varphi'(0)}{(1+\xi_n^{2})^3}.
\end{equation}
 Similarly to (3.18), we have
 \begin{equation}
\pi^+_{\xi_n}\left[\partial_{x_n}\sigma_{-2}(D^{-2})\right](x_0)|_{|\xi'|=1}=\frac{(i\xi_n^{2}+2)\varphi'(0)-i\xi_n\psi'(0)}{4(\xi_n-i)^{2}}.
\end{equation}
Combining (4.6) and (4.8), we obtain

\begin{eqnarray}
 &&\int^{+\infty}_{-\infty}
\frac{(i\xi_n^{2}+2)\varphi'(0)-i\xi_n\psi'(0)}{4(\xi_n-i)^{2}}\times
 \frac{-2+6\xi_n^{2}}{(1+\xi_n^{2})^3} d\xi_n\nonumber\\
&=&-\frac{1}{2}\int_{\Gamma^+} \frac{(3\xi_n^{2}-1)\big(-2\varphi'(0)-i\xi_n\varphi'(0)+i\xi_n\psi'(0)\big)}
{(\xi_n-i)^5(\xi_n+i)^3}d\xi_ndx'\nonumber\\
&=&-\frac{1}{2}\pi i\Big[\frac{(3\xi_n^{2}-1)\big(-2\varphi'(0)-i\xi_n\varphi'(0)+i\xi_n\psi'(0)\big)}
{(\xi_n+i)^3}\Big]^{(4)}|_{\xi_n=i}dx'\nonumber\\
&=&\frac{1}{32}\big(5\varphi'(0)+\psi'(0)\big),
\end{eqnarray}
Since $n=6$, ${\rm tr}_{S(TM)}[{\rm id}]={\rm dim}(\wedge^*(3))=8.$
So by (4.5),(4.9), we get
 \begin{equation}
{\rm case~ a)~II)}=-\frac{1}{8}\big(5\varphi'(0)+\psi'(0)\big)\pi \Omega_4dx'.
\end{equation}
 where $\Omega_4$ is the canonical volume of $S^4$.

{\bf case a)~III)}~$r=-2,~l=-2,~j=|\alpha|=0,~k=1$\\

From (2.4)  and an integration by parts, we get
\begin{eqnarray}
{\rm case~ a)~III)}&=&-\frac{1}{2}\int_{|\xi'|=1}\int^{+\infty}_{-\infty}
{\rm trace} [\partial_{\xi_n}\pi^+_{\xi_n}\sigma_{-2}(D^{-2})\times
\partial_{\xi_n}\partial_{x_n}\sigma_{-2}(D^{-2})](x_0)d\xi_n\sigma(\xi')dx'\nonumber\\
&=&\frac{1}{2}\int_{|\xi'|=1}\int^{+\infty}_{-\infty} {\rm trace}
[\partial_{\xi_n}^2\pi^+_{\xi_n}\sigma_{-2}(D^{-2})\times
\partial_{x_n}\sigma_{-2}(D^{-2})](x_0)d\xi_n\sigma(\xi')dx'.
\end{eqnarray}

By Lemma 3.1 and Lemma 3.2, we have
 \begin{equation}
\partial_{\xi_n}^2\pi^+_{\xi_n}\sigma_{-2}(D^{-2})(x_0)|_{|\xi'|=1}=\frac{-i}{(\xi_n-i)^3}.
\end{equation}
Substituting (4.7) and (4.12) into (4.11), one sees that
\begin{eqnarray}
{\rm {\bf case~a)~III)}}&=&
\frac{1}{2}\int_{|\xi'|=1}\int^{+\infty}_{-\infty}\frac{8i\varphi'(0)-8i\xi_n^{2}\psi'(0)}{(\xi_n-i)^5(\xi_n+i)^2}d\xi_n\sigma(\xi')dx' \nonumber\\
&=&\frac{1}{8}\big(5\varphi'(0)+\psi'(0)\big)\pi \Omega_4dx'.
\end{eqnarray}

{\bf case b)}~$r=-2,~l=-3,~k=j=|\alpha|=0$\\

From (2.4)  and an integration by parts, we get
\begin{eqnarray}
{\rm case~ b)}&=&-i\int_{|\xi'|=1}\int^{+\infty}_{-\infty}
{\rm trace} [\pi^+_{\xi_n}\sigma_{-2}(D^{-2})\times
\partial_{\xi_n}\sigma_{-3}(D^{-2})](x_0)d\xi_n\sigma(\xi')dx'\nonumber\\
&=&=i\int_{|\xi'|=1}\int^{+\infty}_{-\infty}
{\rm trace} [\partial_{\xi_n}\pi^+_{\xi_n}\sigma_{-2}(D^{-2})\times
\sigma_{-3}(D^{-2})](x_0)d\xi_n\sigma(\xi')dx'.
\end{eqnarray}
By Lemma 3.2, we have
 \begin{equation}
\partial_{\xi_n}\pi_{\xi_n}^+\sigma_{-2}(D^{-2})(x_0)|_{|\xi'|=1}=\frac{i}{2(\xi_n-i)^2}.
\end{equation}

In the normal coordinate, $g^{ij}(x_0)=\delta_i^j$ and $\partial_{x_j}(g^{\alpha\beta})(x_0)=0,$ {\rm if
}$j<n;~=\varphi'(0) \delta^\alpha_\beta- \psi'(0)\delta^n_n,~{\rm if }~j=n.$ So by Lemma A.2 \cite{Wa3},
we have $\Gamma^n(x_0)=\frac{5}{2}\varphi'(0)+\frac{1}{2}\psi'(0)$ and
$\Gamma^k(x_0)=0$ for $k<n$. By the definition of $\delta^k$ and Lemma 2.3 in \cite{Wa3}, we have $\delta^n(x_0)=0$ and
$\delta^k=\frac{1}{4}\varphi'(0)c(\widetilde{e_k})c(\widetilde{e_n})$ for $k<n$. So
\begin{eqnarray}
&&\sigma_{-3}(D^{-2})(x_0)|_{|\xi'|=1}\nonumber\\
&=&-\sqrt{-1}|\xi|^{-4}\xi_k(\Gamma^k-2\delta^k)(x_0)|_{|\xi'|=1}-\sqrt{-1}|\xi|^{-6}2\xi^j\xi_\alpha\xi_\beta
\partial_jg^{\alpha\beta}(x_0)|_{|\xi'|=1}\nonumber\\
&=&=\frac{-i}{(1+\xi_n^2)^2}\Big(-\frac{1}{2}\varphi'(0)\sum_{k<n}\xi_k
c(\widetilde{e_k})c(\widetilde{e_n})+\xi_n(\frac{5}{2}\varphi'(0)+\frac{1}{2}\psi'(0))\Big)
-\frac{2i\xi_n(\varphi'(0)a- \psi'(0))}{(1+\xi_n^2)^3}.
\end{eqnarray}
We note that $\int_{|\xi'|=1}\xi_1\cdots\xi_{2q+1}\sigma(\xi')=0$,
so the first term in (4.16) has no contribution for computing case
b). Then
\begin{eqnarray}
{\bf case~ b)}&=&
  \int_{|\xi'|=1}\int^{+\infty}_{-\infty}\frac{2i\xi_n\Big(9\varphi'(0)+5\xi_n^{2}\varphi'(0)-3\psi'(0)+\psi'(0)\xi_n^{2}\Big)}
  {(\xi_n-i)^5(\xi_n+i)^3}
  d\xi_n\sigma(\xi')dx' \nonumber\\
 &=&  \Omega_4  \int_{\Gamma^+} \frac{2i\xi_n\Big(9\varphi'(0)+5\xi_n^{2}\varphi'(0)-3\psi'(0)+\psi'(0)\xi_n^{2}\Big)}
  {(\xi_n-i)^5(\xi_n+i)^3}d\xi_ndx'\nonumber\\
&=&-\frac{3}{8}\big(5\varphi'(0)-\psi'(0)\big)\pi \Omega_4dx'.
\end{eqnarray}

 {\bf  case c)}~$r=-3,~l=-2,~k=j=|\alpha|=0$\\

From (2.4) we have
 \begin{equation}
{\rm {\bf case~ c)}}=-i\int_{|\xi'|=1}\int^{+\infty}_{-\infty}{\rm trace} [\pi^+_{\xi_n}\sigma_{-3}(D^{-2})\times
\partial_{\xi_n}\sigma_{-2}(D^{-2})](x_0)d\xi_n\sigma(\xi')dx'.
\end{equation}
By (24) in \cite{Wa4}, we have
 \begin{equation}
{\rm {\bf case~ c)}}={\rm {\bf case~ b)}}-i\int_{|\xi'|=1}\int^{+\infty}_{-\infty}{\rm
tr}[\partial_{\xi_n}\sigma_{-2}(D^{-2})\times
\sigma_{-3}(D^{-2})]d\xi_n\sigma(\xi')dx'.
\end{equation}
From (4.2) we have
 \begin{equation}
\partial_{\xi_n}\sigma_{-2}(D^{-2})(x_0)=\frac{-2\xi_n}{(1+\xi_n^{2})^2}.
\end{equation}
Combining (4.16) and (4.20), we obtain
  \begin{equation}
-i\int_{|\xi'|=1}\int^{+\infty}_{-\infty}{\rm tr}[\partial_{\xi_n}\sigma_{-2}(D^{-2})\times
\sigma_{-3}(D^{-2})]d\xi_n\sigma(\xi')dx'=\frac{3}{4}\big(5\varphi'(0)-\psi'(0)\big)\pi \Omega_4dx'
\end{equation}
From (4.19) and (4.21), we obtain
  \begin{equation}
{\bf case~ c)}=\frac{3}{8}\big(5\varphi'(0)-\psi'(0)\big)\pi \Omega_4dx'.
\end{equation}
Since $\Phi$ is the sum of the cases a), b) and c), so $\Phi=0$. Therefore

\begin{thm}
Let  $M$ be a $6$-dimensional
compact spin manifold with the boundary $\partial M$ and the metric
$g^M$ as above and $D$ be the Dirac operator on $\widehat{M}$, then
\begin{equation}
{\rm Vol}^{(2,2)}_6=\widetilde{{\rm Wres}}[\pi^+D^{-2}\circ\pi^+D^{-2}]=-\frac{5\Omega_6}{3}\int_Ms{\rm dvol}_M.
\end{equation}
\end{thm}

 Now, we recall the Einstein-Hilbert action for manifolds with boundary. Let $\psi'(0)=\varphi'(0)$,
 from case a) II) and case b) in section 4,  then we obtain:
\begin{thm}
Let  $M$ be a $6$-dimensional compact spin manifold with the boundary $\partial M$ and the metric
$g^M$ as above and $D$ be the Dirac operator on $\widehat{M}$, then
\begin{eqnarray}
&& \int_{\partial M}{\rm res}_{2,2}(D^{-2},D^{-2})=\frac{3\pi}{20}\Omega_4I_{\rm {Gr,b}};  \\
&&\int_{\partial M}{\rm res}_{2,3}(D^{-2},D^{-2})=\frac{3\pi}{10}\Omega_4I_{\rm {Gr,b}}.
\end{eqnarray}
\end{thm}

\section{A Kastler-Kalau-Walze type theorem for $6$-dimensional spin manifolds with boundary of  conformal warped product metric associated with $D$
and $D^{3}$ }
 In this section, We compute the lower dimensional volume ${\rm Vol}^{(1,3)}_6$ for $6$-dimensional spin manifolds with
boundary and get a Kastler-Kalau-Walze type theorem in this case.

Firstly, we compute the symbol $\sigma(D^{-3})$ of $D^{-3}$. Recall the definition of the Dirac operator $D$  \cite{FGLS}\cite{Y}.
 Let $\nabla^L$ denote the Levi-Civita connection about $g^M$. In the local coordinates $\{x_i; 1\leq i\leq n\}$ and the
 fixed orthonormal frame $\{\widetilde{e_1},\cdots,\widetilde{e_n}\}$, the connection matrix $(\omega_{s,t})$ is defined by
\begin{equation}
\nabla^L(\widetilde{e_1},\cdots,\widetilde{e_n})= (\widetilde{e_1},\cdots,\widetilde{e_n})(\omega_{s,t}).
\end{equation}

 The Dirac operator is defined by
\begin{equation}
D=\sum^n_{i=1}c(\widetilde{e_i})\Big[\widetilde{e_i}
-\frac{1}{4}\sum_{s,t}\omega_{s,t}(\widetilde{e_i})c(\widetilde{e_s})c(\widetilde{e_t})\Big].
\end{equation}
where $c(\widetilde{e_i})$ denotes the Clifford action.

Recall the definition of the Dirac operator $D^{2}$ in \cite{Ka}, \cite{KW} and \cite{Wa4}, we have
\begin{equation}
D^{2}=-\sum_{i,j}g^{i,j}\Big[\partial_{i}\partial_{j}+2\sigma_{i}\partial_{j}+(\partial_{i}\sigma_{j})+\sigma_{i}\sigma_{j}
    -\Gamma_{i,j}^{k}\partial_{i}-\Gamma_{i,j}^{k}\sigma_{k}\Big]+\frac{1}{4}s.
\end{equation}
where $\sigma_{i}:=-\frac{1}{4}\sum_{s,t}\omega_{s,t}(\partial_{i})e_se_t$.

Combining (5.2) and (5.3), we have
\begin{eqnarray}
D^{3}
&=&\sum^n_{i=1}c(\widetilde{e_i}) \langle e_i, dx_{l}\rangle \bigg\{ -\sum_{i,j}g^{i,j}\partial_{i}\partial_{j}\partial_{l}
 -\sum_{i,j}g^{i,j}(4\sigma_{i}\partial_{j}-2\Gamma_{i,j}^{k}\partial_{k})\partial_{l}\nonumber\\
&&-\sum_{i,j}g^{i,j}\Big(2(\partial_{l}\sigma_{i})\partial_{j}+2(\partial_{i}\sigma_{j})\partial_{l}
 +3\sigma_{i}\sigma_{j}\partial_{l}+3\Gamma_{i,j}^{k}\sigma_{k}\partial_{l}\Big) +\frac{1}{4}s\partial_{l} \nonumber\\
&& -\sum_{i,j}g^{i,j}\Big((\partial_{l}\partial_{i}\sigma_{j})+(\partial_{l}\sigma_{i})\sigma_{j}
+\sigma_{i}(\partial_{l}\sigma_{j})-\Gamma_{i,j}^{k}(\partial_{l}\sigma_{k})\Big)+\frac{1}{4}(\partial_{l}s)\bigg\}   \nonumber\\
&& -\frac{1}{4}\sum_{s,t}\omega_{s,t}(\widetilde{e_i})c(\widetilde{e_i})c(\widetilde{e_s})c(\widetilde{e_t})
   \bigg\{-\sum_{i,j}g^{i,j}\partial_{i}\partial_{j} -\sum_{i,j}g^{i,j}(2\sigma_{i}\partial_{j}-\Gamma_{i,j}^{k}\partial_{k})  \nonumber\\
&& -\sum_{i,j}g^{i,j}\Big((\partial_{i}\sigma_{j})+\sigma_{i}\sigma_{j}-\Gamma_{i,j}^{k}\sigma_{k}\Big)+\frac{1}{4}s\bigg\}.
\end{eqnarray}
By Section in \cite{WW1}, we obtain
\begin{eqnarray}
\sigma_{3}(D^{3})&=&\sqrt{-1}c(\xi)|\xi|^2 , \\
\sigma_{2}(D^{3})&=&c(\xi)(4\sigma^k-2\Gamma^k)\xi_{k}
     -\frac{1}{4}|\xi|^2\sum_{s,t}\omega_{s,t}(\widetilde{e_l})c(e_{l})c(\widetilde{e_s})c(\widetilde{e_t}), \\
\sigma_{1}(D^{3})&=&\sum^n_{i=1}c(\widetilde{e_i}) \langle e_i, dx_{l}\rangle \bigg[
   -\sum_{i,j}g^{i,j}\Big(2(\partial_{l}\sigma_{i})\partial_{j}+2(\partial_{i}\sigma_{j})\partial_{l}
   +3\sigma_{i}\sigma_{j}\partial_{l}+3\Gamma_{i,j}^{k}\sigma_{k}\partial_{l}\Big) +\frac{1}{4}s\partial_{l}
    \bigg]   \nonumber\\
   && -\frac{1}{4}\sum_{s,t}\omega_{s,t}(\widetilde{e_i})c(\widetilde{e_i})c(\widetilde{e_s})c(\widetilde{e_t})
   \bigg[-\sum_{i,j}g^{i,j}(2\sigma_{i}\partial_{j}-\Gamma_{i,j}^{k}\partial_{k})
     \bigg],\\
\sigma_{0}(D^{3})&=&\sum^n_{i=1}c(\widetilde{e_i}) \langle e_i, dx_{l}\rangle
     \bigg[ -\sum_{i,j}g^{i,j}\Big((\partial_{l}\partial_{i}\sigma_{j})+(\partial_{l}\sigma_{i})\sigma_{j}
      +\sigma_{i}(\partial_{l}\sigma_{j})-\Gamma_{i,j}^{k}(\partial_{l}\sigma_{k})\Big)+\frac{1}{4}(\partial_{l}s)
     \bigg]   \nonumber\\
    && -\frac{1}{4}\sum_{s,t}\omega_{s,t}(\widetilde{e_i})c(\widetilde{e_i})c(\widetilde{e_s})c(\widetilde{e_t})
   \bigg[ -\sum_{i,j}g^{i,j}\Big((\partial_{i}\sigma_{j})+\sigma_{i}\sigma_{j}-\Gamma_{i,j}^{k}\sigma_{k}\Big)+\frac{1}{4}s\bigg].
\end{eqnarray}

Write
 \begin{equation}
D_x^{\alpha}=(-\sqrt{-1})^{|\alpha|}\partial_x^{\alpha};
~\sigma(D^{3})=p_3+p_2+p_1+p_0;
~\sigma(D^{-3})=\sum^{\infty}_{j=3}q_{-j}.
\end{equation}
 By the composition formula of psudodifferential operators, we have
 \begin{eqnarray}
1=\sigma(D^{3}\circ D^{-3})&=&\sum_{\alpha}\frac{1}{\alpha!}\partial^{\alpha}_{\xi}[\sigma(D)]D^{\alpha}_{x}[\sigma(D^{-1})] \nonumber\\
&=&(p_3+p_2+p_1+p_0)(q_{-3}+q_{-4}+q_{-5}+\cdots) \nonumber\\
&&+\sum_j(\partial_{\xi_j}p_3+\partial_{\xi_j}p_2+\partial_{\xi_j}p_1+\partial_{\xi_j}p_0)
(D_{x_j}q_{-3}+D_{x_j}q_{-4}+D_{x_j}q_{-5}+\cdots) \nonumber\\
&=&p_3q_{-3}+(p_3q_{-4}+p_2q_{-3}+\sum_j\partial_{\xi_j}p_3D_{x_j}q_{-3})+\cdots,
\end{eqnarray}
Then we obtain
\begin{equation}
q_{-3}=p_3^{-1};~q_{-4}=-p_3^{-1}[p_2p_3^{-1}+\sum_j\partial_{\xi_j}p_3D_{x_j}(p_3^{-1})].
\end{equation}
By Lemma 2.1 in \cite{Wa3} and (5.4)-(5.11), we obtain

\begin{lem}
\begin{eqnarray}
\sigma_{-3}(D^{-3})&=&\frac{\sqrt{-1}c(\xi)}{|\xi|^4}; \\
\sigma_{-4}(D^{-3})&=&\frac{c(\xi)\sigma_{2}(D^{3})c(\xi)}{|\xi|^8}
+\frac{c(\xi)}{|\xi|^{10}}\sum_j\Big[c(dx_j)|\xi|^2+2\xi_{j}c(\xi)\Big]\Big[\partial_{x_j}[c(\xi)]|\xi|^2-2c(\xi)\partial_{x_j}(|\xi|^2)\Big],
\end{eqnarray}
where $\sigma_{0}(D)=-\frac{1}{4}\sum_{s,t}\omega_{s,t}(\widetilde{e_i})c(\widetilde{e_i})c(\widetilde{e_s})c(\widetilde{e_t})$.
\end{lem}

 Since $\Phi$ is a global form on $\partial M$, so for any fixed point $x_0\in\partial M$, we can choose the normal coordinates
$U$ of $x_0$ in $\partial M$ (not in $M$) and compute $\Phi(x_0)$ in the coordinates $\widetilde{U}=U\times [0,1)\subset M$ and the
metric $ g^M=\frac{1}{\varphi(x_n)}g^{\partial M}+\psi(x_n)dx_n^2$. The dual metric of $g^M$ on $\widetilde{U}$ is
$\varphi(x_n) g^{\partial M}+\frac{1}{\psi(x_n)}dx_n^2.$
Write $g^M_{ij}=g^M(\frac{\partial}{\partial x_i},\frac{\partial}{\partial x_j});~ g_M^{ij}=g^M(dx_i,dx_j)$, then

\begin{equation}
[g^M_{i,j}]= \Big[\begin{array}{lcr}
  \frac{1}{\varphi(x_n)}[g_{i,j}^{\partial M}]  & 0  \\
   0  &  \psi(x_n)
\end{array}\Big];~~~
[g_M^{i,j}]= \Big[\begin{array}{lcr}
  \varphi(x_n)[g^{i,j}_{\partial M}]  & 0  \\
   0  &  \frac{1}{\psi(x_n)}
\end{array}\Big],
\end{equation}
and
 \begin{equation}
\partial_{x_s}g_{ij}^{\partial M}(x_0)=0, 1\leq i,j\leq n-1; ~~~g_{ij}^M(x_0)=\delta_{ij}.
\end{equation}

Let $n=6$ and $\{e_1,\cdots,e_{n-1}\}$ be an orthonormal frame field in $U$ about $g^{\partial M}$ which is parallel along geodesics and
$e_i(x_0)=\frac{\partial}{\partial x_i}(x_0)$, then $\{\widetilde{e_1}=\sqrt{\varphi(x_n)}e_1,\cdots,
\widetilde{e_{n-1}}=\sqrt{\varphi(x_n)}e_{n-1},
\widetilde{e_n}=\frac{1}{\sqrt{\psi(x_n)}}dx_n\}$ is the orthonormal frame field in $\widetilde U$ about $g^M$. Locally $S(TM)|_{\widetilde {U}}\cong
\widetilde {U}\times\wedge^* _{\bf C}(\frac{n}{2}).$
Let $\{f_1,\cdots,f_8\}$ be the orthonormal basis of $\wedge^* _{\bf C}(\frac{n}{2})$.
Take a spin frame field $\sigma:~\widetilde {U}\rightarrow {\rm Spin}(M)$ such that $\pi\sigma=
\{\widetilde{e_1},\cdots,\widetilde{e_n}\}$, where $\pi :~{\rm Spin}(M)\rightarrow O(M)$ is a double covering,
 then $\{[(\sigma,f_i)],~1\leq i\leq 8\}$ is an orthonormal frame of $S(TM)|_{\widetilde {U}}.$ In the following, since the global form
$\Phi$ is independent of the choice of the local frame, so we can compute ${\rm tr}_{S(TM)}$ in the frame $\{[(\sigma,f_i)],~1\leq
i\leq 8\}.$ Let $\{E_1,\cdots,E_n\}$ be the canonical basis of
${\bf R}^n$ and $c(E_i)\in {\rm cl}_{\bf C}(n)\cong {\rm Hom}(\wedge^*_{\bf C}(\frac{n}{2}),\wedge^* _{\bf C}(\frac{n}{2}))$
be the Clifford action. By \cite{Wa3}, \cite{Wa4} and \cite{Y} , we have
\begin{equation}
c( \widetilde{e_i})=[(\sigma,c(E_i))];~ c( \widetilde{e_i})[(\sigma,f_i)]=[(\sigma,c(E_i)f_i)];~
\frac{\partial}{\partial x_i}=[(\sigma,\frac{\partial}{\partial
x_i})],
\end{equation}
then we have $\frac{\partial}{\partial x_i}c( \widetilde{e_i})=0$ in the above frame.

Combining (3.3) and Lemma 3.3, we obtain

\begin{lem} For $4$-dimensional spin manifolds with boundary,
 \begin{equation}
\sigma_{0}(D)(x_0)=-\frac{5}{4} \varphi'(0)c(dx_n).
\end{equation}
\end{lem}

 Now we can compute $\Phi$ (see formula (2.4) for the definition of $\Phi$), since the sum is taken over $
-r-l+k+j+|\alpha|=5,~~r\leq -1,~l\leq -4,$ then we have the following five cases:

{\bf case a)~I)}~$r=-1,~l=-3,~k=j=0,~|\alpha|=1$

From (2.4) we have
 \begin{equation}
{\rm case~a)~I)}=-\int_{|\xi'|=1}\int^{+\infty}_{-\infty}\sum_{|\alpha|=1}
 {\rm trace}[\partial^\alpha_{\xi'}\pi^+_{\xi_n}\sigma_{-1}(D^{-1})\times
 \partial^\alpha_{x'}\partial_{\xi_n}\sigma_{-3}(D^{-3})](x_0)d\xi_n\sigma(\xi')dx'.
\end{equation}
By Lemma 3.2, for $i<n$, then
\begin{eqnarray}
\partial_{x_i}\sigma_{-1}(D^{-1})(x_0)&=&\partial_{x_i}\left(\frac{\sqrt{-1}c(\xi)}{|\xi|^2}\right)(x_0) \nonumber\\
&=&\frac{\sqrt{-1}\partial_{x_i}[c(\xi)](x_0)}{|\xi|^2}
-\frac{\sqrt{-1}c(\xi)\partial_{x_i}(|\xi|^2)(x_0)}{|\xi|^4}\nonumber\\
&=&0.
\end{eqnarray}
Then case a) I) vanishes.

 {\bf case a)~II)}~$r=-1,~l=-3, ~k=|\alpha|=0,~j=1$

From (2.4) we have
 \begin{equation}
{\rm case~ a)~II)}=-\frac{1}{2}\int_{|\xi'|=1}\int^{+\infty}_{-\infty} {\rm
trace} [\partial_{x_n}\pi^+_{\xi_n}\sigma_{-1}(D^{-1})\times
\partial_{\xi_n}^2\sigma_{-3}(D^{-3})](x_0)d\xi_n\sigma(\xi')dx'.
\end{equation}
By Lemma 3.1 and Lemma 3.2, we have
 \begin{equation}
\partial^2_{\xi_n}\sigma_{-3}(D^{-3})=\frac{20i\xi_n^{2}-4i}{(1+\xi_n^{2})^{4}}c(\xi')+\frac{12i\xi_n^{3}-12i\xi_n}{(1+\xi_n^{2})^{4}}c(dx_n).
\end{equation}
Since $n=6$, ${\rm tr}_{S(TM)}[{\rm id}]={\rm dim}(\wedge^*(3))=8$.
By the relation of the Clifford action and ${\rm tr}{AB}={\rm tr }{BA}$, we have the equalities:
\begin{eqnarray}
&&{\rm tr}[c(\xi')c(dx_n)]=0;~~{\rm tr}[c(dx_n)^2]=-8;~~{\rm tr}[c(\xi')^2](x_0)|_{|\xi'|=1}=-8; \nonumber\\
&&{\rm tr}[\partial_{x_n}[c(\xi')]c(dx_n)]=0;~~{\rm tr}[\partial_{x_n}[c(\xi')]c(\xi')](x_0)|_{|\xi'|=1}=-4\varphi'(0);\nonumber\\
&&{\rm tr}[\partial_{x_n}[c(dx_n)]c(\xi')]=0;~~{\rm tr}[\partial_{x_n}[c(dx_n)]c(dx_n)](x_0)|_{|\xi'|=1}=4\psi'(0).
\end{eqnarray}
From (3.22), (5.21)and (5.22), we have
\begin{eqnarray}
&&{\rm tr}\Big\{\Big[\frac{\partial_{x_n}[c(\xi')](x_0)}{2(\xi_n-i)}+\frac{i\partial_{x_n}[c(dx_n)](x_0)}{2(\xi_n-i)}
+\frac{(2i-\xi_n)\varphi'(0)+\xi_n\psi'(0)}{4(\xi_n-i)^{2}}c(\xi')
+\frac{-\varphi'(0)+(1+2i\xi_n)\psi'(0)}{4(\xi_n-i)^2}c(dx_n)\Big]\nonumber\\
&&~~~~\times
\Big[\frac{20i\xi_n^{2}-4i}{(1+\xi_n^{2})^{4}}c(\xi')+\frac{12i\xi_n^{3}-12i\xi_n}{(1+\xi_n^{2})^{4}}c(dx_n)\Big]
\Big\}(x_0)|_{|\xi'|=1} \nonumber\\
&=&\frac{8(i\varphi'(0)+\xi_n\psi'(0))(-1-2 i \xi_n+3\xi_n^{2})}{(\xi_n-i)^5(\xi_n+i)^4}.
\end{eqnarray}
Substituting (5.23) into (5.20), we have
\begin{eqnarray}
{\rm case~ a)~II)}&=&-\frac{1}{2}\int_{|\xi'|=1}\int^{+\infty}_{-\infty}
 \frac{8(i\varphi'(0)+\xi_n\psi'(0))(-1-2 i \xi_n+3\xi_n^{2})}{(\xi_n-i)^5(\xi_n+i)^4}d\xi_n\sigma(\xi')dx' \nonumber\\
&=&-\Omega_4\int_{\Gamma^+}\frac{4(i\varphi'(0)+\xi_n\psi'(0))(-1-2 i \xi_n+3\xi_n^{2})}{(\xi_n+i)^4}d\xi_ndx'\nonumber\\
&=&-\Omega_4 \frac{2\pi i}{4!}\Big[\frac{4(i\varphi'(0)+\xi_n\psi'(0))(-1-2 i \xi_n+3\xi_n^{2})}{(\xi_n+i)^4}\Big]^{(4)}|_{\xi_n=i}dx'\nonumber\\
&=&-\frac{1}{16}\big(15\varphi'(0)+7\psi'(0)\big)\pi \Omega_4dx',
\end{eqnarray}
 where $\Omega_4$ is the canonical volume of $S^4$.

{\bf case a)~III)}~$r=-3,~l=-1,~j=|\alpha|=0,~k=1$\\

From (2.4) we have
 \begin{equation}
{\rm case~ a)~III)}=-\frac{1}{2}\int_{|\xi'|=1}\int^{+\infty}_{-\infty}
{\rm trace} [\partial_{\xi_n}\pi^+_{\xi_n}\sigma_{-1}(D^{-1})\times
\partial_{\xi_n}\partial_{x_n}\sigma_{-3}(D^{-3})](x_0)d\xi_n\sigma(\xi')dx'.
\end{equation}
By (2.2.29) in \cite{Wa3}, we have
 \begin{equation}
\partial_{\xi_n}\pi^+_{\xi_n}\sigma_{-1}(D^{-1})(x_0)|_{|\xi'|=1}=\frac{-c(\xi')}{2(\xi_n-i)^2}+\frac{-ic(dx_n)}{2(\xi_n-i)^2}.
\end{equation}
From (5.12) we have
\begin{eqnarray}
\partial_{\xi_n}\partial_{x_n}\sigma_{-3}(D^{-3})(x_0)|_{|\xi'|=1}
&=&\frac{-4i\xi_n}{(1+\xi_n^{2})^{3}}\partial_{x_n}[c(\xi')](x_0)+\frac{i-3i\xi_n^{2}}{(1+\xi_n^{2})^{3}}\partial_{x_n}[c(dx_n)](x_0)\nonumber\\
&&+\frac{(10i\xi_n^{2}-2i)\varphi'(0)+(6i\xi_n^{2}-6i\xi_n^{4})\psi'(0)}{(1+\xi_n^{2})^{4}}c(dx_n)\nonumber\\
&&+\frac{12i\xi_n\varphi'(0)+(4i\xi_n-8i\xi_n^{3})\psi'(0)}{(1+\xi_n^{2})^{4}}c(\xi').
\end{eqnarray}
Combining (5.26) and (5.27), we obtain
\begin{eqnarray}
&&{\rm  tr}\Big\{\Big[\frac{-c(\xi')-ic(dx_n)}{2(\xi_n-i)^2}\Big]\times
\Big[\frac{-4i\xi_n}{(1+\xi_n^{2})^{3}}\partial_{x_n}[c(\xi')](x_0)+\frac{i-3i\xi_n^{2}}{(1+\xi_n^{2})^{3}}\partial_{x_n}[c(dx_n)](x_0)\nonumber\\
&&~~~+\frac{(10i\xi_n^{2}-2i)\varphi'(0)+(6i\xi_n^{2}-6i\xi_n^{4})\psi'(0)}{(1+\xi_n^{2})^{4}}c(dx_n)
+\frac{12i\xi_n\varphi'(0)+(4i\xi_n-8i\xi_n^{3})\psi'(0)}{(1+\xi_n^{2})^{4}}c(\xi')
\Big]
\Big\}(x_0)|_{|\xi'|=1} \nonumber\\
&=&\frac{(8i-32\xi_n-8i\xi_n^{2} )\varphi'(0)+(2i-14\xi_n-14i\xi_n^{2}+18\xi_n^{3})\psi'(0)}{(\xi_n-i)^5(\xi_n+i)^4}.
\end{eqnarray}
Substituting (5.28) into (5.25), one sees that
\begin{eqnarray}
{\rm {\bf case~a)~III)}}&=&
-\frac{1}{2}\int_{|\xi'|=1}\int^{+\infty}_{-\infty}
\frac{(8i-32\xi_n-8i\xi_n^{2} )\varphi'(0)+(2i-14\xi_n-14i\xi_n^{2}+18\xi_n^{3})\psi'(0)}{(\xi_n-i)^5(\xi_n+i)^4}d\xi_n\sigma(\xi')dx' \nonumber\\
&=& -\frac{1}{2}\times\frac{2 \pi i}{4!}\Big[\frac{(8i-32\xi_n-8i\xi_n^{2} )\varphi'(0)+(2i-14\xi_n-14i\xi_n^{2}+18\xi_n^{3})\psi'(0)}
 {(\xi_n+i)^4} \Big]^{(4)}|_{\xi_n=i}\Omega_4dx'\nonumber\\
&=&\frac{1}{16}\big(25\varphi'(0)+\psi'(0)\big)\pi \Omega_4dx'.
\end{eqnarray}

{\bf case b)}~$r=-2,~l=-3,~k=j=|\alpha|=0$\\

From (2.4) we have
 \begin{equation}
{\rm case~ b)}=-i\int_{|\xi'|=1}\int^{+\infty}_{-\infty}{\rm trace} [\pi^+_{\xi_n}\sigma_{-2}(D^{-1})\times
\partial_{\xi_n}\sigma_{-3}(D^{-3})](x_0)d\xi_n\sigma(\xi')dx'.
\end{equation}
By Lemma 3.1 , Lemma 3.2 and Lemma 5.1, we have
\begin{equation}
\partial_{\xi_n}\sigma_{-3}(D^{-3})=\frac{-4 i \xi_n c(\xi')+(i-3i\xi_n^{2})c(dx_n)}{(1+\xi_n^{2})^3}.
\end{equation}
and
\begin{eqnarray}
\sigma_{-2}(D^{-1})(x_0)&=& \frac{c(\xi)\sigma_{0}(D)(x_0)c(\xi)}{|\xi|^4}
+\frac{1}{|\xi|^4}c(\xi)c(dx_n)\Big[\partial_{x_n}[c(\xi')](x_0)+\xi_n\partial_{x_n}[c(dx_n)](x_0)\Big]  \nonumber\\
&&-\frac{1}{|\xi|^6}c(\xi)c(dx_n)c(\xi)\Big[\varphi'(0)-\xi_n^{2}\psi'(0)\Big].
\end{eqnarray}
Then
\begin{eqnarray}
\pi^+_{\xi_n}\sigma_{-2}(D^{-1})(x_0)|_{|\xi'|=1}&=& \pi^+_{\xi_n}\Big[\frac{c(\xi)\sigma_{0}(D)(x_0)c(\xi)}{|\xi|^4}\Big]\nonumber\\
&&+\pi^+_{\xi_n}\Big[\frac{1}{|\xi|^4}c(\xi)c(dx_n)\big[\partial_{x_n}[c(\xi')](x_0)+\xi_n\partial_{x_n}[c(dx_n)](x_0)\big]\Big]  \nonumber\\
&&+\pi^+_{\xi_n}\Big[\frac{1}{|\xi|^6}c(\xi)c(dx_n)c(\xi)\big[\xi_n^{2}\psi'(0)-\varphi'(0)\big]\Big] \nonumber\\
&:=&A+B+C.
\end{eqnarray}
 Similarly to (3.18), by Lemma 3.4 we have
 \begin{equation}
A=\frac{-1}{4(\xi_n-i)^{2}}\Big[\frac{-3\varphi'(0)}{4}(2+i\xi_n)c(\xi')c(dx_n)c(\xi')-\frac{-3i\xi_n\varphi'(0)}{4} c(dx_n)
+\frac{3i\varphi'(0)}{2} c(\xi')\Big].
\end{equation}
And
\begin{eqnarray}
B&=&\frac{-1}{4(\xi_n-i)^{2}}\Big[(2+i\xi_n)c(\xi')c(dx_n)\partial_{x_n}[c(\xi')](x_0)
+ic(\xi')c(dx_n)\partial_{x_n}[c(dx_n)](x_0) \nonumber\\
&&-i\partial_{x_n}[c(\xi')](x_0)
-i\xi_n\partial_{x_n}[c(dx_n)](x_0) \Big],
\end{eqnarray}
\begin{eqnarray}
C&=&\frac{\xi_n}{16(\xi_n-i)^{3}}\Big[-3\varphi'(0)-i\xi_n\varphi'(0)+\psi'(0)+3i\xi_n\psi'(0)  \Big]c(dx_n) \nonumber\\
  &&+\frac{1}{8(\xi_n-i)^{3}}\Big[-3\varphi'(0)-i\xi_n\varphi'(0)+\psi'(0)+3i\xi_n\psi'(0)  \Big]c(\xi')\nonumber\\
  &&+\frac{1}{16(\xi_n-i)^{3}}\Big[-8i\varphi'(0)+9\xi_n \varphi'(0)+3i\xi_n^{2}\varphi'(0)-3\xi_n\psi'(0)-i\xi_n^{2}\psi'(0)
        \Big]c(\xi')c(dx_n)c(\xi').
\end{eqnarray}
 By (5.31) and (5.34)-(5.36), we obtain
\begin{eqnarray}
{\rm tr }[A\times\partial_{\xi_n}\sigma_{-3}(D^{-3})(x_0)]|_{|\xi'|=1}&=&\frac{(5+15 i\xi_n)\varphi'(0)}{(\xi_n-i)^{4}(\xi_n+i)^{3}};\\
{\rm tr }[B\times\partial_{\xi_n}\sigma_{-3}(D^{-3})(x_0)]|_{|\xi'|=1}&=&
 \frac{(-2-3i\xi_n+3\xi_n^{2})\varphi'(0)+(3i\xi_n+3\xi_n^{2})\psi'(0)}{(\xi_n-i)^{4}(\xi_n+i)^{3}};
\\
{\rm tr }[C\times\partial_{\xi_n}\sigma_{-3}(D^{-3})(x_0)]|_{|\xi'|=1}&=&
\frac{(-4i+11\xi_n+6i\xi_n^{2}-3\xi_n^{3})\varphi'(0)-(5\xi_n+6i\xi_n^{2}+3\xi_n^{3})\psi'(0)}{(\xi_n-i)^{5}(\xi_n+i)^{3}}.
\end{eqnarray}
 Combining (5.30), (5.37), (5.38) and (5.39), we obtain
\begin{eqnarray}
{\bf case~ b)}&=&
 -i \int_{|\xi'|=1}\int^{+\infty}_{-\infty}\frac{(3\xi_n-i)(-7i\varphi'(0)+5\xi_n\varphi'(0)-2\xi_n\psi'(0))}{-i(\xi_n-i)^{5}(\xi_n+i)^{3}}
  d\xi_n\sigma(\xi')dx' \nonumber\\
 &=& -i \Omega_4 \int_{\Gamma^+} \frac{(3\xi_n-i)(-7i\varphi'(0)+5\xi_n\varphi'(0)-2\xi_n\psi'(0))}{-i(\xi_n-i)^{5}(\xi_n+i)^{3}}d\xi_ndx'\nonumber\\
&=& -i \Omega_4 \frac{2 \pi i}{4!}\Big[\frac{(3\xi_n-i)(-7i\varphi'(0)+5\xi_n\varphi'(0)-2\xi_n\psi'(0))}{-i(\xi_n+i)^{3}}
     \Big]^{(4)}|_{\xi_n=i}dx'\nonumber\\
&=&\frac{1}{16}\big(55\varphi'(0)-\psi'(0)\big)\pi \Omega_4dx'.
\end{eqnarray}

 {\bf  case c)}~$r=-1,~l=-4,~k=j=|\alpha|=0$\\

From (2.4) we have
 \begin{equation}
{\rm case~ c)}=-i\int_{|\xi'|=1}\int^{+\infty}_{-\infty}{\rm trace} [\pi^+_{\xi_n}\sigma_{-1}(D^{-1})\times
\partial_{\xi_n}\sigma_{-4}(D^{-3})](x_0)d\xi_n\sigma(\xi')dx'.
\end{equation}
By (2.2.44) in \cite{Wa3}, we have
 \begin{equation}
\pi^+_{\xi_n}\sigma_{-1}(D^{-1})(x_0)|_{|\xi'|=1}=\frac{c(\xi')+ic(dx_n)}{2(\xi_n-i)};
\end{equation}
From (5.13) we have
\begin{eqnarray}
&&\partial_{\xi_n}\sigma_{-4}(D^{-3})(x_0)\nonumber\\
 &=&\frac{(59\xi_{n}+27\xi_{n}^{3})\varphi'(0)+(8\xi_{n} -24\xi_{n}^{3})\psi'(0) }{2(1+\xi_{n}^2)^{5}}c(\xi')c(dx_n)c(\xi')\nonumber\\
  &&+\frac{(33-180\xi_{n}^{2}-85\xi_{n}^{4})\varphi'(0) +(-48\xi_{n}^{2}+80\xi_{n}^{4})\psi'(0)}{2(1+\xi_{n}^2)^{5}}c(\xi')\nonumber\\
 && +\frac{(49\xi_{n}-97\xi_{n}^3-50\xi_{n}^{5})\varphi'(0)+(-48\xi_{n}^{3}+48\xi_{n}^{5})\psi'(0)  }{2(1+\xi_{n}^2)^{5}}c(dx_n)\nonumber\\
 &&+\frac{-6\xi_{n}}{(1+\xi_{n}^2)^{4}}c(\xi')c(dx_n)\partial_{x_n}[c(\xi')](x_0)
 +\frac{-3+15\xi_{n}^{2}}{(1+\xi_{n}^2)^{4}}\partial_{x_n}[c(\xi')](x_0)\nonumber\\
&&+\frac{1-5\xi_{n}^{2}}{(1+\xi_{n}^2)^{4}}c(\xi')c(dx_n)\partial_{x_n}[c(dx_n)](x_0)
 +\frac{-6\xi_{n}+12\xi_{n}^{3}}{(1+\xi_{n}^2)^{4}}\partial_{x_n}[c(dx_n)](x_0).
\end{eqnarray}
Then similarly to computations of the case b), we have
\begin{eqnarray}
&&{\rm trace} [\pi^+_{\xi_n}\sigma_{-1}(D^{-1})\times \partial_{\xi_n}\sigma_{-4}(D^{-3})](x_0)|_{|\xi'|=1}\nonumber\\
&=&\frac{2i}{(\xi_n-i)^{5}(\xi_n+i)^{5}}
\Big[(-30-72i\xi_n+96\xi_n^{2}-20 i\xi_n^{3}+50\xi_n^{4} )\varphi'(0)  \nonumber\\
&&+ (1-15i\xi_n+29\xi_n^{2}+49i\xi_n^{3}-36\xi_n^{4}) \psi'(0)\Big]
\end{eqnarray}
Combining (5.41) and (5.44), we obtain
  \begin{equation}
{\bf case~ c)}=-\frac{3}{16}\big(35\varphi'(0)-6\psi'(0)\big)\pi \Omega_4dx'.
\end{equation}

Since $\Phi$ is the sum of the cases a), b) and c), so
  \begin{equation}
\Phi=-\frac{1}{16}\big(40\varphi'(0)-11\psi'(0)\big)\pi \Omega_4dx'.
\end{equation}

Now recall the Einstein-Hilbert action for manifolds with boundary \cite{Wa3}\cite{Wa4}\cite{Y},
 \begin{equation}
I_{\rm Gr}=\frac{1}{16\pi}\int_Ms{\rm dvol}_M+2\int_{\partial M}K{\rm dvol}_{\partial_M}:=I_{\rm {Gr,i}}+I_{\rm {Gr,b}},
\end{equation}
 where
 \begin{equation}
 K=\sum_{1\leq i,j\leq {n-1}}K_{i,j}g_{\partial M}^{i,j};~~K_{i,j}=-\Gamma^n_{i,j},
\end{equation}
 and $K_{i,j}$ is the second fundamental form, or extrinsic curvature. Taking the metric in Section 2, then by Lemma
A.2 \cite{Wa3}, for $n=6$, then
 \begin{equation}
K(x_0)=-\frac{5}{2}\varphi'(0);~
 I_{\rm {Gr,b}}=-5\varphi'(0){\rm Vol}_{\partial M}.
\end{equation}
Let $\psi'(0)=c\varphi'(0)$, then we obtain

\begin{thm}
Let  $M$ be a $6$-dimensional
compact spin manifold with the boundary $\partial M$ and the metric
$g^M$ as above, $D$ be the Dirac operator on $\widehat{M}$ and $\psi'(0)=c\varphi'(0)$, then
\begin{equation}
{\rm Vol}^{(1,3)}_6=\widetilde{{\rm Wres}}[\pi^+D^{-1}\circ\pi^+D^{-3}]=-\frac{5\Omega_5}{3}\int_Ms{\rm dvol}_M
 +(1-\frac{11c}{40})\pi \Omega_4\int_{\partial M}Kd{\rm Vol}_{\partial M}.
\end{equation}
\end{thm}

\begin{rem}
In \cite{Wa3} \cite{Wa4}, Wang computed
 $\widetilde{{\rm Wres}}[\pi^+D^{-1}\circ\pi^+D^{-1}]$ and $\widetilde{{\rm Wres}}[\pi^+D^{-2}\circ\pi^+D^{-2}]$.
 In that cases, the boundary terms vanished, where the two operators are symmetric.  Theorem 5.3 states the
 boundary terms is non-zero when we compute $\widetilde{{\rm Wres}}[\pi^+D^{-1}\circ\pi^+D^{-3}]$. The reason
 is that $D^{-1}$ and $D^{-3}$ are not  symmetric.
\end{rem}

Let
 \begin{equation}
\widetilde{{\rm Wres}}[\pi^+D^{-1}\circ\pi^+D^{-3}]
=\widetilde{{\rm Wres}}_{i}[\pi^+D^{-1}\circ\pi^+D^{-3}]+\widetilde{{\rm Wres}}_{b}[\pi^+D^{-1}\circ\pi^+D^{-3}],
\end{equation}
where
 \begin{equation}
\widetilde{{\rm Wres}}_{i}[\pi^+D^{-1}\circ\pi^+D^{-3}]
=\int_M\int_{|\xi|=1}{\rm trace}_{S(TM)}[\sigma_{-6}(D^{-1-3})]\sigma(\xi)dx
\end{equation}
and
\begin{eqnarray}
&&\widetilde{{\rm Wres}}_{b}[\pi^+D^{-1}\circ\pi^+D^{-3}]\nonumber\\
&=& \int_{\partial M}\int_{|\xi'|=1}\int^{+\infty}_{-\infty}\sum^{\infty}_{j, k=0}\sum\frac{(-i)^{|\alpha|+j+k+1}}{\alpha!(j+k+1)!}
\times {\rm trace}_{S(TM)}[\partial^j_{x_n}\partial^\alpha_{\xi'}\partial^k_{\xi_n}\sigma^+_{r}(D^{-1})(x',0,\xi',\xi_n)
\nonumber\\
&&\times\partial^\alpha_{x'}\partial^{j+1}_{\xi_n}\partial^k_{x_n}\sigma_{l}(D^{-3})(x',0,\xi',\xi_n)]d\xi_n\sigma(\xi')dx'
\end{eqnarray}
denote the interior term and boundary term of $\widetilde{{\rm Wres}}[\pi^+D^{-1}\circ\pi^+D^{-3}]$.

Combining (5.47), (5.50) and (5.51), we obtain
\begin{cor}
Let  $M$ be a $6$-dimensional
compact spin manifold with the boundary $\partial M$ and the metric
$g^M$ as above and $D$ be the Dirac operator on $\widehat{M}$, then
\begin{eqnarray}
&& I_{\rm {Gr,i}}=\frac{-3}{80 \pi\Omega_5 }\widetilde{{\rm Wres}}_{i}[\pi^+D^{-1}\circ\pi^+D^{-3}]; \nonumber\\
&&I_{\rm {Gr,b}}=\frac{80}{(40-11c)\pi\Omega_4 }\widetilde{{\rm Wres}}_{b}[\pi^+D^{-1}\circ\pi^+D^{-3}].
\end{eqnarray}
\end{cor}

 In particular, when $c=\frac{40}{11}+\frac{6400\Omega_5}{33\Omega_4}$, then we obtain
\begin{thm}
Let  $M$ be a $6$-dimensional
compact spin manifold with the boundary $\partial M$ and the metric
$g^M$ as above, $D$ be the Dirac operator on $\widehat{M}$ and $\psi'(0)=c\varphi'(0)$, then
\begin{eqnarray}
&& {\rm Vol}^{(1,3)}_6=\widetilde{{\rm Wres}}[\pi^+D^{-1}\circ\pi^+D^{-3}]=-\frac{80\pi\Omega_5}{3}\Big[\frac{1}{16\pi}\int_Ms{\rm dvol}_M
 +2\int_{\partial M}Kd{\rm Vol}_{\partial M}\Big]; \nonumber\\
&&I_{\rm {Gr}}=\frac{-3}{80 \pi\Omega_5 }\widetilde{{\rm Wres}}[\pi^+D^{-1}\circ\pi^+D^{-3}].
\end{eqnarray}
\end{thm}

\section*{ Acknowledgements}
This work was supported by Fok Ying Tong Education Foundation under Grant No. 121003 and NSFC. 11271062. And Jian Wang's Email
address: wangj068@gmail.com. The author also thank the referee for his (or her) careful reading and helpful comments.

\end{document}